\documentclass[12pt]{article}

\usepackage{graphicx}
\usepackage{amssymb}
\usepackage{amsthm}
\usepackage{amsmath}
\usepackage{euscript}
\usepackage{color}
\usepackage{tikz-cd}
\usepackage{stmaryrd}
\usepackage{fullpage}
\usepackage[colors]{optsys}

\newcommand{\preference}{\precsim}
\newcommand{\wstrategy}{\policy}
\newcommand{\WSTRATEGY}{\POLICY}

\newcommand{\SetValuedReducedSolutionMap}{{\cal M}}
\newcommand{\ReducedSolutionMap}{M}
\newcommand{\bgent}{b}

\newcommand{\BGENT}{\mathbb{B}}

\renewcommand{\cardinal}[1]{\left| #1 \right|}
\renewcommand{\range}[1]{\|#1\|}

\newcommand{\LastElement}[1]{{#1}^\star}
\newcommand{\FirstElements}[1]{{#1}^-}

\newcommand{\Choice}{C}

\newcommand{\atom}{G}

\renewcommand{\AGENT}{{\mathbb A}}
\newcommand{\cut}{\psi}
\newcommand{\ORDER}{\Sigma}
\newcommand{\ordering}{\varphi}
\newcommand{\totalordering}{\rho}

\newcommand{\numAGENT}{| \AGENT |}

\newcommand{\SET}{{\mathbb D}}
\newcommand{\Set}{D}

\newcommand{\quotient}[2]{\raise1ex\hbox{$#1$}\Big/\lower1ex\hbox{$#2$}}

\newcommand{\BehavioralWStrategy}{\beta}

\newcommand{\MixedStrategy}{\mu}
\newcommand{\ProductMixedStrategy}{\pi}
\newcommand{\probability}{\nu}
\newcommand{\conditionaly}[2]{\np{\na{#1}\,\vert\,{#2}}}
\newcommand{\conditionalySet}[2]{\np{{#1}\,\vert\,{#2}}}

\newcommand{\PMtoB}[1]{\hat{#1}}
\newcommand{\BtoPM}[1]{\check{#1}}
\newcommand{\MtoB}[1]{\tilde{#1}}

\renewcommand{\defset}[3][]{\ba{#2\,;\, #3}}
\renewcommand{\crochet}[1]{\left< #1 \right>}
\renewcommand{\bracket}[1]{\left[ #1 \right]} 
\renewcommand{\sequence}[2]{\left(#1\right)_{#2}}           

\newcommand{\ProductBehavioral}{\ProductMixedStrategy}

\renewcommand{\borel}[1]{\tribu{B}_{#1}}

\newenvironment{keyword}{\bgroup
  \hsize=\textwidth%
  \noindent\unskip\textbf{Keywords.}\noindent\,\ignorespaces}
 {\egroup}
\title{Kuhn's Equivalence Theorem\\ for Games in Intrinsic Form}
\author{Benjamin Heymann\footnote{Criteo, Paris, France}, 
  Michel De Lara\footnote{CERMICS, Ecole des Ponts, Marne-la-Vall\'ee, France},
  Jean-Philippe Chancelier$^\dagger$}
\date{\today}

\begin{document}

\maketitle


\begin{abstract}
   We state and prove Kuhn's equivalence theorem
  for a new representation of  games, the intrinsic form.
First, we introduce games in intrinsic form where 
information is represented by $\sigma$-fields over a product set.
For this purpose, we adapt to  games the intrinsic representation
 that Witsenhausen  introduced in control theory.
Those intrinsic games do not require an explicit description of the play
 temporality, as opposed to extensive form
 games on trees.   
Second, we prove, for this new and more general representation of games, 
that behavioral and mixed strategies are equivalent under perfect recall
(Kuhn's theorem).
As the intrinsic form replaces the tree structure with a product structure, 
the handling of information is easier.
This makes the intrinsic form a new valuable tool for the analysis of
games with information.
\end{abstract}



\section{Introduction}

From the origin, games in extensive form have been formulated on a tree.
In his seminal paper
\emph{Extensive Games and the Problem of Information}~\cite{Kuhn:1953},
Kuhn claimed that ``The use of a geometrical model  (\ldots)
clarifies the delicate problem of information''. This tells us that  the proper handling of information was a strong 
motivation for Kuhn's extensive games. 
On the game tree, moves are those vertices that possess alternatives,
then moves are partitioned into players moves, themselves partitioned into 
information sets (with the constraint that no two moves in an 
information set can be on the same play).
Kuhn mentions agents, one agent per information set, 
to ``personalize the interpretation'' but the notion is not central
(to the point that his definition of perfect recall
``obviates the use of agents'').

By contrast, in the so-called Witsenhausen's intrinsic model
\cite{Witsenhausen:1971a,Witsenhausen:1975},
agents play a central role. 
Each agent is equipped with a decision set and a $\sigma$-field,
and the same for Nature. Then, Witsenhausen introduces the product set
and the product $\sigma$-field. This product set hosts the agents' information subfields. 
The Witsenhausen's intrinsic model was elaborated 
in the control theory setting, in order to handle how information
is distributed among agents 
and how it impacts their strategies.
Although not explicitly designed for games, 
Witsenhausen's intrinsic model had, from the start, the potential 
to be adapted to games. Indeed, in~\cite{Witsenhausen:1971a} Witsenhausen 
 places his own model in the context of game theory by referring to 
von~Neuman and Morgenstern~\cite{vonNeuman-Morgenstern:1947},
Kuhn~\cite{Kuhn:1953} and Aumann~\cite{Aumann:1964}.
\medskip

In this paper, we introduce a new representation of games
that we call \emph{games in intrinsic form}. 
Game representations play a key role in their analysis
(see the illuminating introduction of the book~\cite{Alos-Ferrer-Ritzberger:2016}),
and we claim that games in intrinsic form display appealing features.
In the philosophy of the tree-based extensive form (Kuhn's view), 
the temporal ordering is 
hardcoded in the tree structure: one goes from the root to the leaves,
making decisions at the moves, contingent on information, chance and strategies.
For Kuhn, the time arrow (tree) comes first; information comes second
(partition of the move vertices).
By contrast, for Witsenhausen, information comes first;
the time arrow comes (possibly) second, 
under a proper causality assumption contingent to the information structure. 

Not having a hardcoded temporal ordering makes mathematical representations
less constrained, hence more general.
Moreover, Witsenhausen's framework makes representations more intrinsic.
As an illustration, let us consider a game where two players play once but at the same
time. To formulate it on a tree requires to arbitrarily decide which of the
two plays first. This is not the case for games in intrinsic form,
where each player/agent is equipped with an information subfield
and strategies that are measurable with respect to the latter;
writing the system of two equations that express decisions as 
the output of strategies leads to a unique outcome,
without having to solve one equation first, and the other second.

The tree representation of games has its pros and cons.
On the one hand, trees are perfect to follow step by step how a game is played
as any strategy profile induces a unique play:
one goes from the root to the leaves, passing from one node
to the next by an edge that depends on the strategy profile.
On the other hand, in games with information,
  information sets are represented as ``union'' of tree nodes
that must satisfy restrictive axioms, and such
  unions do not comply in a natural way with the tree structure,
  which can render the game analysis delicate
  \cite{Alos-Ferrer-Ritzberger:2016,Bonanno:2004,Brandenburger:2007}.
By contrast, the notion of Witsenhausen's intrinsic games (W-games) does not require
an explicit description of the play temporality, and the intrinsic form replaces
the tree structure with a product structure, more amenable to mathematical analysis. 
If the introduction of the model may seem involved, we argue that the
resulting structure is a powerful mathematical tool, because there are many
situations in which it is easier to reason and discuss with mathematical
formulas than with trees.

We illustrate our claim with a proof of the celebrated Kuhn's equivalence theorem for games
in intrinsic form.
Indeed, as a first step in a broader research program, 
 we show that equivalence between mixed and behavioral
 strategies holds under perfect recall for W-games.
More precisely, our proof relies on an equivalence between behavioral,
mixed and a new notion of product-mixed strategies.
These latter form a subclass of mixed strategies.
In the spirit of~\cite{Aumann:1964}, 
in a product-mixed strategy, each agent (corresponding to time index in~\cite{Aumann:1964}) 
generates strategies from a random device that is independent of all the
other agents. We prove that, under perfect recall for W-games,
any mixed strategy of a player is not only equivalent to a behavioral strategy,
but also to a product-mixed strategy where all the agents under control 
of the player randomly select their pure strategy independently of the other agents.
\medskip

The paper is organized as follows.
In Sect.~\ref{Witsenhausen_intrinsic_model},
we present the finite version of Witsenhausen's intrinsic model. 
Then, in Sect.~\ref{Finite_games_in_intrinsic_form}, we propose a formal
definition of games in intrinsic form (W-games), and then discuss three notions of 
``randomization'' of pure strategies --- mixed, product-mixed and behavioral.
Finally, we derive an equivalent of Kuhn's equivalence theorem for games in intrinsic form
in Sect.~\ref{Kuhn_Theorem}.
In Appendix~\ref{Background_on_fields_atoms_and_partitions}, 
we present background material on fields, atoms and partitions,
as these notions lay at the core of Witsenhausen's intrinsic model
in the finite case.
In all the paper, we adopt the convention that a player is female (hence using ``she'' and
``her''), whereas an agent is male (``he'', ``his'').

\section{Witsenhausen's intrinsic model (the finite case)}
\label{Witsenhausen_intrinsic_model}

In this paper, we tackle the issue of information in the context of finite games.
For this purpose, we will present the so-called intrinsic model of Witsenhausen
\cite{Witsenhausen:1975,Carpentier-Chancelier-Cohen-DeLara:2015}
but with finite sets rather than with infinite ones 
as in the original exposition.
We refer the reader to 
Appendix~\ref{Background_on_fields_atoms_and_partitions}
for background material on fields, atoms and partitions.

In~\S\ref{Finite_Witsenhausen_intrinsic_model},
we present the finite version of Witsenhausen's 
intrinsic model, where we highlight the role of the configuration field
that contains the information subfields of all agents.
In~\S\ref{Examples}, we illustrate, on a few examples, the ease with which 
one can model information in strategic contexts, 
using subfields of the configuration field.
Finally, we present in~\S\ref{Solvability_Causality} 
the notions of solvability and causality. 

\subsection{Finite Witsenhausen's intrinsic model (W-model)}
\label{Finite_Witsenhausen_intrinsic_model}

We present the finite version of Witsenhausen's 
intrinsic model, introduced some five decades ago in the control community
\cite{Witsenhausen:1971a,Witsenhausen:1975}. 

\begin{definition}(adapted from \cite{Witsenhausen:1971a,Witsenhausen:1975})

A \emph{finite W-model} is a collection
$\bp{
\AGENT,
\np{\Omega, \tribu{\NatureField}}, 
\sequence{\CONTROL_{\agent}, \tribu{\Control}_{\agent}}{\agent \in \AGENT}, 
\sequence{\tribu{\Information}_{\agent}}{\agent \in \AGENT} 
}$, where 
\begin{itemize}
\item
$\AGENT$ is a finite set, whose elements are called \emph{agents};
\item 
\( \Omega \) is a finite set which represents all uncertainties;
any $\omega \in \Omega$ is called a \emph{state of Nature};
$\tribu{\NatureField}$ is the complete field over~\( \Omega \);
\item 
for any \( \agent \in \AGENT \), $\CONTROL_{\agent}$ is a finite set, 
the \emph{set of decisions} for agent~$\agent$;
$\tribu{\Control}_{\agent}$ is the complete field over~$\CONTROL_{\agent}$;
\item
for any \( \agent \in \AGENT \), \( \tribu{\Information}_{\agent} \)
is a subfield of the following product field
\begin{equation}
  \tribu{\Information}_{\agent} \subset 
  {\oproduit{\bigotimes \limits_{\bgent \in \AGENT} \tribu{\Control}_{\bgent}}{\tribu{\NatureField}}}
  \eqsepv
  \forall \agent \in \AGENT 
  \label{eq:information_field_agent}
\end{equation}

and is called the \emph{information field} of the agent~$\agent$.
\end{itemize}
\label{de:W-model}
\end{definition}

\begin{subequations}
The \emph{configuration space} is the product space 
(called \emph{hybrid space} by Witsenhausen, hence the $\HISTORY$ notation)
\begin{equation}
  \label{eq:HISTORY}
  \HISTORY = \produit{ \prod\limits_{\agent \in \AGENT} \CONTROL_{\agent}}{\Omega} 
  \eqfinp 
\end{equation}
As all fields \( \tribu{\NatureField} \) and
\( \sequence{\tribu{\Control}_{\agent}}{\agent \in \AGENT} \) are complete, 
the product \emph{configuration field}
\begin{equation}
  \tribu{\History} =\oproduit{{\bigotimes \limits_{\agent \in \AGENT} \tribu{\Control}_{\agent}}}{\tribu{\NatureField}}
  \label{eq:history_field}
\end{equation}
is also the complete field of~\( \HISTORY \).
A \emph{configuration} \( \history \in \HISTORY \) is denoted by
\begin{equation}
  \history=\bp{\omega, \sequence{\control_{\agent}}{\agent \in \AGENT}}
  \iff
  \history_\emptyset = \omega
 \text{ and }
  \history_{\agent} =\control_{\agent}
  \eqsepv
  \forall \agent \in \AGENT
  \eqfinp
  \label{eq:history}
\end{equation}
\end{subequations}

In lieu of the information field~\( \tribu{\Information}_{\agent} \)
in~\eqref{eq:information_field_agent},
it will be convenient to consider the equivalence relation~$\sim_{\agent}$,
on the configuration space~$\HISTORY$, defined in such a way that 
the equivalence classes \( \bracket{\cdot}_{\agent} \subset \HISTORY \) 
coincide with the atoms of~\( \tribu{\Information}_{\agent} \),
that is, with the elements of the partition
\( \crochet{\tribu{\Information}_{\agent}} \) in~\eqref{eq:atom_set}:
\begin{equation}
  \bp{\forall \history', \history'' \in \HISTORY} \quad
  \history' \sim_{\agent} \history'' \;\Leftrightarrow\; 
  \history'' \in \bracket{\history'}_{\agent}\;\Leftrightarrow\; 
  \exists \atom \in \crochet{\tribu{\Information}_{\agent}},
  \, \{ \history', \history'' \} \subset \atom
  \eqfinp
  \label{eq:bracket_HISTORY}
\end{equation}
Thus defined, the subset \( \bracket{\history}_{\agent} \subset \HISTORY \)
is the unique atom~$\atom$ in \( \crochet{\tribu{\Information}_{\agent}} 
\subset \tribu{\Information}_{\agent} \)
that contains the configuration~$\history$.

We will need the following equivalent characterization of measurable
mappings, which is a slight reformulation of~\cite[Proposition~3.35]{Carpentier-Chancelier-Cohen-DeLara:2015}.

\begin{proposition}(adapted from \cite[Proposition~3.35]{Carpentier-Chancelier-Cohen-DeLara:2015})
  \label{pr:measurability}
  Let  \( \rho : (\HISTORY,\tribu{\History}) \to (\SET,\tribu{\Set}) \) 
  be a mapping, 
  where $\SET$ is a set and $\tribu{\Set}$ is a $\sigma$-field over~$\SET$.
  We suppose that the $\sigma$-field~$\tribu{\Set}$ contains all the singletons.
  Then, for any agent~$\agent \in \AGENT$, the following statements are equivalent:
  \begin{subequations}
    \begin{equation}
      \rho^{-1} (\tribu{\Set}) \subset \tribu{\Information}_{\agent} 
      \eqfinv
    \end{equation}
    \begin{equation}
      \bp{ \forall \history', \history'' \in \HISTORY } \quad
      \history' \sim_{\agent} \history'' \implies 
      \rho\np{\history'}=\rho\np{\history''}
      \eqfinv
      \label{eq:strategy_atoms_rho}
    \end{equation}
    \begin{equation}
      \begin{split}
        \textrm{the set-valued mapping } \hat\rho :
        \crochet{\tribu{\Information}_{\agent}}\rightrightarrows\SET 
        \eqsepv \textrm{defined by} \\
        \hat\rho\np{\atom_{\agent}}= \nset{ \rho\np{\history}}%
        { \history \in \atom_{\agent}}
        \eqsepv
        \forall \atom_{\agent} \in \crochet{\tribu{\Information}_{\agent}},
        \textrm{ is a mapping.}
      \end{split}
      \label{eq:measurability-and-set-valued-map}
    \end{equation}
  \end{subequations}
  In any of these equivalent cases, we say that 
  the mapping~\( \rho \) is \emph{$\tribu{\Information}_{\agent}$-measurable},
  and, for all $\atom_{\agent}\in \crochet{\tribu{\Information}_{\agent}}$, we denote by \( \rho\np{\atom_{\agent}} \)
  the unique element of~$\SET$ in $\hat\rho\np{\atom_{\agent}}$, that is, 
  \begin{equation}
    \bp{ \forall r \in \rho\np{\HISTORY}
      \eqsepv \forall \atom_{\agent} \in \crochet{\tribu{\Information}_{\agent}} }\quad
    \rho\np{\atom_{\agent}} =r \iff
    \hat\rho\np{{\atom_{\agent}}} = \na{r}
    \eqfinp
    \label{eq:common_value_rho}
  \end{equation}
  Then, using the extended notation above~\eqref{eq:bracket_HISTORY}, we have the property 
  \begin{equation}
    \rho \text{ is }\tribu{\Information}_{\agent}\text{-measurable}
    \implies 
    \rho\np{ \bracket{\history}_{\agent} } =
    \rho\np{\history} 
    \eqsepv \forall \history \in \HISTORY
    \eqfinp
    \label{eq:common_value_equivalence_class_rho}
  \end{equation}
\end{proposition}

Now that we have explicited measurable mappings with respect to 
agents information subfields, we introduce the notion of 
pure W-strategy.

\begin{definition}(\cite{Witsenhausen:1971a,Witsenhausen:1975})
  \label{de:W-strategy}
  A \emph{pure W-strategy} of agent~$\agent \in \AGENT$ is a mapping
  \begin{subequations}
    \begin{equation}
      \wstrategy_{\agent} : (\HISTORY,\tribu{\History}) \to
      (\CONTROL_{\agent},\tribu{\Control}_{\agent})
    \end{equation}
    from configurations to decisions,
    which is measurable with respect to the information
    field~$\tribu{\Information}_{\agent}$ of agent~$\agent$, that is,
    \begin{equation}
      \label{eq:decision_rule}
      \wstrategy_{\agent}^{-1} (\tribu{\Control}_{\agent})
      \subset \tribu{\Information}_{\agent} 
      \eqfinp
    \end{equation}
  \end{subequations}
  \begin{subequations}
    We denote by $\WSTRATEGY_{\agent}$ the set of all pure W-strategies of agent $\agent \in \AGENT$.
    A \emph{pure W-strategies profile}~$\wstrategy$ is a family 
    \begin{equation}
      \wstrategy = \sequence{\wstrategy_{\agent}}{\agent \in \AGENT} 
      \in \prod_{\agent \in \AGENT} \WSTRATEGY_{\agent} 
      \label{eq:W-strategy_profile}
    \end{equation}
    of pure W-strategies, one per agent~$\agent \in \AGENT$.
    The \emph{set of pure W-strategies profiles} is 
    \begin{equation}
      \WSTRATEGY= \prod_{\agent \in \AGENT} \WSTRATEGY_{\agent} 
      \eqfinp
      \label{eq:W-STRATEGY} 
    \end{equation}
  \end{subequations}
\end{definition}
Condition~\eqref{eq:decision_rule} expresses the property that any 
(pure) W-strategy of agent~$\agent$ may only depend upon the
information~$\tribu{\Information}_{\agent}$ available to the agent. 

In what follows, we will need some notations.
For any nonempty subset $\BGENT \subset \AGENT$ of agents, we define 
\begin{subequations}
  \begin{align}
    \tribu{\Control}_\BGENT 
    &= 
      \bigotimes \limits_{\bgent \in \BGENT} \tribu{\Control}_\bgent
      \otimes
      \bigotimes \limits_{\agent \not\in \BGENT} \{ \emptyset, \CONTROL_{\agent} \}
      \subset
      \bigotimes \limits_{\agent \in \AGENT} \tribu{\Control}_{\agent}
      \eqfinv
      \label{eq:sub_control_field_BGENT}
    \\
    \tribu{\History}_\BGENT 
    &= 
      \tribu{\NatureField} \otimes \tribu{\Control}_\BGENT
      = \tribu{\NatureField} \otimes \bigotimes \limits_{\bgent \in \BGENT} \tribu{\Control}_\bgent
      \otimes
      \bigotimes \limits_{\agent \not\in \BGENT} \{ \emptyset, \CONTROL_{\agent} \}
      \subset \tribu{\History}
      \eqfinv
      \label{eq:sub_history_field_BGENT}
    \\
    \history_\BGENT 
    &=
      \sequence{\history_\bgent}{\bgent \in \BGENT}
      \in \prod \limits_{\bgent \in \BGENT} \CONTROL_\bgent
      \eqsepv \forall \history \in \HISTORY
      \eqfinv
      \label{eq:sub_history_BGENT}
    \\
    \wstrategy_\BGENT 
    &=
      \sequence{\wstrategy_\bgent}{\bgent \in \BGENT}
      \in \prod \limits_{\bgent \in \BGENT} \WSTRATEGY_{\bgent} 
      \eqsepv \forall \wstrategy \in \WSTRATEGY
      \eqfinp
      \label{eq:sub_wstrategy_BGENT}
  \end{align}
  %
\end{subequations}

\subsection{Examples}
\label{Examples}

We illustrate, on a few examples, the ease with which 
one can model information in strategic contexts, 
using subfields of the configuration field.
Even if we have presented the finite version of Witsenhausen's intrinsic model
in~\S\ref{Finite_Witsenhausen_intrinsic_model},
we take the opportunity here to show its potential to describe 
infinite decision and Nature sets.

\subsubsubsection{Sequential decisions}
Suppose an individual has to take decisions (say, an element of $\RR^n$) 
at every discrete time step in the set\footnote{For any integers $a \leq b$, $\ic{a,b}$ denotes the subset
  $\na{a,a+1,\ldots,b-1,b}$.} 
\( \ic{1,\horizon{-}1} \), where $\horizon \geq 1$ is an integer.
The situation will be modeled with (possibly) Nature set and field
\( \np{\Omega, \tribu{\NatureField}} \), 
and with $\horizon$~agents in $\AGENT=\ic{0,\horizon{-}1}$, 
and their corresponding sets, $\CONTROL_t= \RR^n$, and fields,
$\tribu{\Control}_t = \borel{\RR^{n}} $
(the Borel $\sigma$-field of~$\RR^n$), for $t \in \AGENT$.
Then, one builds up the product set 
$\HISTORY=\produit{\prod_{t=0}^{\horizon{-}1} \CONTROL_{t} }{\Omega}$ and 
the product field $\tribu{\History}= \oproduit{%
  \bigotimes_{t=0}^{\horizon{-}1} \tribu{\Control}_{t} }{\tribu{\NatureField}}$.
Every agent \( t \in \ic{0,\horizon{-}1} \) is equipped with an 
information field \( \tribu{\Information}_{t} \subset \tribu{\History} \).
Then, we show how we can express four information patterns:
sequentiality, memory of past information, 
memory of past actions, perfect recall.
The inclusions \( \tribu{\Information}_{t} \subset 
\tribu{\History}_{\{0,\ldots,t{-}1\}} = \oproduit{%
\bigotimes_{s=0}^{t-1} \tribu{\Control}_{s} \otimes 
\bigotimes_{s=t}^{\horizon{-}1} \{ \emptyset, \CONTROL_{s} \}}{\tribu{\NatureField}} \),
for \( t\in \ic{0,\horizon{-}1} \), 
express that every agent can remember no more than his past actions
(sequentiality); 
memory of past information is represented by the inclusions 
\(\tribu{\Information}_{t-1} \subset \tribu{\Information}_{t} \),
for \( t\in \ic{1,\horizon{-}1} \);
memory of past actions is represented by the inclusions 
\(  \oproduit{\bigotimes_{s=0}^{t-1} \tribu{\Control}_{s} \otimes 
\bigotimes_{s=t}^{\horizon{-}1} \{ \emptyset, \CONTROL_{s} \}}%
{ \{ \emptyset, \Omega \} } = 
\oproduit{ \tribu{\Control}_{\{0,\ldots,t{-}1\}}}%
{ \{ \emptyset, \Omega \} } 
\subset \tribu{\Information}_{t} \),
for \( t\in \ic{1,\horizon{-}1} \);
perfect recall is represented by the inclusions 
\( \tribu{\Information}_{t-1} \vee \bp{ \oproduit{ \tribu{\Control}_{\{0,\ldots,t{-}1\}}}%
{ \{ \emptyset, \Omega \} } }
\subset \tribu{\Information}_{t} \),
for \( t\in \ic{1,\horizon{-}1} \).

To represent $N$~players --- each~$\player$ of whom makes a sequence of decisions,
one for each period~$t \in \ic{0,\horizon_\player{-}1}$ --- 
we use $\prod_{\player=1}^N \horizon_\player$~agents, labelled by
\( (\player,t) \in \bigcup_{\player'=1}^N \bp{ \na{\player'} \mathord{\times} \ic{0,\horizon_{\player'}{-}1}} \).
With obvious notations, the inclusions 
\(\tribu{\Information}_{\np{\player,t-1}} \subset \tribu{\Information}_{\np{\player,t}} \)
express memory of one's own past information, 
whereas the inclusions 
\( \bigvee_{\player'=1}^N
\bp{\oproduit{\bigotimes_{s=0}^{t-1} \tribu{\Control}_{s} \otimes 
    \bigotimes_{s=t}^{\horizon_{\player'}-1} \{ \emptyset, \CONTROL_{s} \}}
  { \{ \emptyset, \Omega \} }}
\subset \tribu{\Information}_{\np{\player,t}} \),
express memory of all players past actions.

\subsubsubsection{Principal-Agent models} 

A branch of Economics studies so-called \emph{Prin\-ci\-pal-Agent} models with 
two decision makers (agents) --- 
  the Principal~$\Principal$ (\emph{leader}) who makes decisions $\control_{\Principal} \in \CONTROL_{\Principal}$,
  where the set~$\CONTROL_{\Principal}$ is equipped with a
  $\sigma$-field~$\tribu{\Control}_{\Principal}$,
  and the Agent~$\Agent$ (\emph{follower}) who makes decisions $\control_{\Agent} \in \CONTROL_{\Agent}$,
  where the set~$\CONTROL_{\Agent}$ is equipped with a
  $\sigma$-field~$\tribu{\Control}_{\Agent}$ ---
and with Nature, corresponding to \emph{private information (or type)} of the
Agent~$\Agent$, taking values in a set~$\Omega$, 
equipped with a $\sigma$-field~$\tribu{\NatureField}$.

\emph{Hidden type} (leading to adverse selection or to signaling) is represented by any information structure 
with the property that, on the one hand,
\begin{equation}
  \tribu{\Information}_{\Principal} \subset 
  \oproduit{ 
    \{\emptyset,\CONTROL_{\Principal}\} 
    \otimes 
    \underbrace{\tribu{\Control}_{\Agent} }_{\textrm{\makebox[0pt]{\hspace{1cm}$\Agent$'s action possibly observed}}}
  }%
  {\underbrace{ \{\emptyset,\Omega \}  }_{\textrm{\makebox[0pt]{$\Agent$ type not observed}}}}
  \eqfinv  
\end{equation}
that is, 
the Principal~$\Principal$ does not know the Agent~$\Agent$ type,
but can possibly observe the Agent~$\Agent$ action,
and, on the other hand, that
\begin{equation}
  \oproduit{ 
    \{\emptyset,\CONTROL_{\Principal}\} 
    \otimes
    \{\emptyset,\CONTROL_{\Agent}\} 
  }{%
    \underbrace{ \tribu{\NatureField} }_{ \textrm{\makebox[0pt]{known inner type}}}}
  \subset \tribu{\Information}_{\Agent}
  \eqfinv  
  \label{eq:Agent_Information_type}
\end{equation}
that is, the Agent~$\Agent$ knows the state of nature (his type).

\emph{Hidden action} (leading to moral hazard) is represented by any information structure 
with the property that, on the one hand, 
\begin{equation}
  \tribu{\Information}_{\Principal} \subset 
  \oproduit{ 
    \{\emptyset,\CONTROL_{\Principal}\} 
    \otimes 
    \underbrace{ \{\emptyset,\CONTROL_{\Agent}\} }_{\textrm{\makebox[0pt]{\hspace{2cm}cannot observe $\Agent$'s action}}}
  }{%
    \underbrace{ \tribu{\NatureField} }_{ \textrm{\makebox[0pt]{\hspace{1cm}possibly knows $\Agent$ type} } }}
  \eqfinv  
\end{equation}
that is, the Principal~$\Principal$ does not know the Agent~$\Agent$ action,
but can possibly observe the Agent~$\Agent$ type
and, on the other hand, that
the inclusion~\eqref{eq:Agent_Information_type} holds true, 
that is, the agent~$\Agent$ knows the state of nature (his type).

\subsubsubsection{Stackelberg leadership model}
In Stackelberg games, the leader~$\Principal$ makes a decision
\( \control_{\Principal} \in \CONTROL_{\Principal} \)
--- based at most upon the partial observation of the state 
\( \omega \in \Omega \) of Nature ---
and the the follower~$\Agent$ makes a decision
\( \control_{\Agent} \in \CONTROL_{\Agent} \)
--- based at most upon the partial observation of the state of Nature 
\( \omega \in \Omega \), and upon the leader decision
\( \control_{\Principal} \in \CONTROL_{\Principal} \).
This kind of information structure is expressed with the following inclusions 
of fields:
\begin{equation}
\tribu{\Information}_{\Principal} 
\subset 
\oproduit{%
\{\emptyset,\CONTROL_{\Principal}\}
\otimes
\{\emptyset,\CONTROL_{\Agent}\} 
}{%
 \tribu{\NatureField} }
\mtext{ and }
\tribu{\Information}_{\Agent} 
\subset 
\oproduit{%
\tribu{\Control}_{\Principal} 
\otimes
\{\emptyset,\CONTROL_{\Agent}\} 
}{%
 \tribu{\NatureField} }
\eqfinp 
\label{eq:Stackelberg_Information}  
\end{equation}
Even if the players are called leader and follower, 
there is no explicit time arrow in~\eqref{eq:Stackelberg_Information}.
It is the information structure that reveals the time arrow.
Indeed, if we label the leader~$\Principal$ as~$t=0$ (first player)
and the follower~$\Agent$ as~$t=1$ (second player), 
the inclusions~\eqref{eq:Stackelberg_Information} become the inclusions 
\( \tribu{\Information}_{0} \subset \oproduit{%
\{ \emptyset, \CONTROL_{0} \} \otimes 
\{ \emptyset, \CONTROL_{1} \}}{\tribu{\NatureField}} \),
and 
\( \tribu{\Information}_{1} \subset \oproduit{%
\tribu{\Control}_{0} \otimes 
\{ \emptyset, \CONTROL_{1} \}}{\tribu{\NatureField}} \):
the sequence \( \tribu{\Information}_{0}, \tribu{\Information}_{1} \)
of information fields 
is ``adapted'' to the filtration
\( \oproduit{%
\{ \emptyset, \CONTROL_{0} \} \otimes 
\{ \emptyset, \CONTROL_{1} \}}{\tribu{\NatureField}}
\subset \oproduit{%
\tribu{\Control}_{0} \otimes 
\{ \emptyset, \CONTROL_{1} \}}{\tribu{\NatureField}} \).
But if we label the leader~$\Principal$ as~$t=1$ 
and the follower~$\Agent$ as~$t=0$,
the new sequence of information fields would not be 
``adapted'' to the new filtration.
It is the information structure that prevents 
the follower to play first, but that makes possible
the leader to play first and the follower to play second.

\subsection{Solvability and causality}
\label{Solvability_Causality}

In the Kuhn formulation, Witsenhausen says that ``For any
combination of policies one can find the corresponding outcome by
following the tree along selected branches, and this is an explicit
procedure'' \cite{Witsenhausen:1971a}. 
In the Witsenhausen formulation, there is no such explicit procedure
as, for any combination of policies, there may be none, one or many solutions to 
the closed-loop equations; these equations express the decision of one agent as 
the output of his strategy, supplied with Nature outcome and with all agents decisions.
This is why Witsenhausen needs a property of solvability,
whereas Kuhn does not need it as it is hardcoded in the tree structure.
Then, Witsenhausen defines the notion of causality (which parallels that of
tree) and proves in~\cite{Witsenhausen:1971a} that solvability holds true under causality.
Yet, in~\cite[Theorem~2]{Witsenhausen:1971a},
Witsenhausen exhibits an example of noncausal W-model that is solvable.

\subsubsection{Solvability}

With any given pure W-strategies profile $\wstrategy = \sequence{\wstrategy_{\agent}}{\agent \in \AGENT}
\in \prod_{\agent \in \AGENT} \WSTRATEGY_{\agent} $ we associate the set-valued
mapping
\begin{align}
\SetValuedReducedSolutionMap_\wstrategy: \Omega
  & \rightrightarrows \prod_{\bgent \in \AGENT} \CONTROL_{\bgent} 
\label{eq:SetValuedReducedSolutionMap}
  \\
  \omega
  & \mapsto \Bset{ \sequence{\control_\bgent}{\bgent \in \AGENT} \in 
    \prod_{\bgent \in \AGENT} \CONTROL_{\bgent} }{%
\control_{\agent} = \wstrategy_{\agent}\bp{\omega,
\sequence{\control_\bgent}{\bgent \in \AGENT} }
    \eqsepv \forall \agent \in \AGENT}
    \eqfinp
    \nonumber
\end{align}
%

With this definition, we slightly reformulate below
how Witsenhausen introduced the property of solvability.

\begin{definition}(\cite{Witsenhausen:1971a,Witsenhausen:1975})
  \begin{subequations}
    The \emph{solvability property} holds true for the W-model of Definition~\ref{de:W-model}
    when,
    for any pure W-strategies profile $\wstrategy = \sequence{\wstrategy_{\agent}}{\agent \in \AGENT}
    \in \prod_{\agent \in \AGENT} \WSTRATEGY_{\agent} $, 
the set-valued mapping~$\SetValuedReducedSolutionMap_{\wstrategy}$
in~\eqref{eq:SetValuedReducedSolutionMap}
 is a mapping whose domain is $\Omega$, that is,
    the cardinal of $\SetValuedReducedSolutionMap_{\wstrategy}\np{\omega}$ 
is equal to one, for any state of nature $\omega \in \Omega$. 

    Thus, under the solvability property, for any state of nature $\omega \in \Omega$, 
    there exists one, and only one, decision profile
    $\sequence{\control_\bgent}{\bgent \in \AGENT} \in 
    \prod_{\bgent \in \AGENT} \CONTROL_{\bgent}$
    which is a solution of the \emph{closed-loop equations}
    \begin{equation}
      \control_{\agent} = \wstrategy_{\agent}\bp{\omega, \sequence{\control_\bgent}{\bgent \in \AGENT}}
      \eqsepv
      \forall \agent \in \AGENT
      \eqfinp
      \label{eq:solution_map_IFF}
    \end{equation}
    In this case, we define the \emph{solution map} 
    \begin{equation}
      \ReducedSolutionMap_\wstrategy: \Omega \rightarrow \prod_{\bgent \in \AGENT} \CONTROL_{\bgent}
      \label{eq:solution_map}
    \end{equation}
    as the unique element contained in the image set
$\SetValuedReducedSolutionMap_{\wstrategy}\np{\omega}$ that is, 
for all $\sequence{\control_\bgent}{\bgent \in \AGENT} \in 
    \prod_{\bgent \in \AGENT} \CONTROL_{\bgent}$, 
    $\ReducedSolutionMap_\wstrategy\np{\omega} = 
\sequence{\control_\bgent}{\bgent \in \AGENT} \iff
    \SetValuedReducedSolutionMap_{\wstrategy}\np{\omega} = \na{\sequence{\control_\bgent}{\bgent \in \AGENT}}$.
%
    %
  \end{subequations}
  \label{de:solvability}
\end{definition}

\subsubsection{Configuration-orderings}

In his articles \cite{Witsenhausen:1971a,Witsenhausen:1975},
Witsenhausen introduces a notion of causality that relies on 
suitable configuration-orderings.
Here, we introduce our own notations, because they make possible a compact formulation 
of the causality property and, later, of perfect recall.
\medskip

\begin{subequations}
For any finite set~\( \SET \), let \( \cardinal{\SET} \) denote the cardinal
of~\( \SET \). 
  Thus, \( \numAGENT \) denotes the cardinal of the set~\( \AGENT \), that is,
\( \numAGENT \) is the number of agents.
For $k \in \ic{1,\numAGENT}$, let $\ORDER^k$ denote the set of
$k$-orderings, that is, injective
mappings from $\ic{1, k}$ to $\AGENT$:
\begin{equation}
  \ORDER^k=\defset{ \kappa: \ic{1,k} \to \AGENT }%
{ \kappa \mtext{ is an injection} }
\eqfinp 
\label{eq:ORDER_k}
\end{equation}
The set \( \ORDER^{\numAGENT} \) is the set of \emph{total orderings} of
agents in $\AGENT$, that is, bijective
mappings from $\ic{1,\numAGENT}$ to $\AGENT$
(in contrast with \emph{partial orderings} in~$\ORDER^k$ for $k < \numAGENT$).
For any $k \in \ic{1, \numAGENT}$, any ordering $\kappa \in \ORDER^k$,
and any integer $\ell \le k$, \( \kappa_{\vert \{ 1, \ldots, \ell \}} \)
is the restriction of the ordering~$\kappa$ to the first $\ell$~integers.
For any $k \in \ic{1,\numAGENT}$, there is a natural mapping $\cut_k$ 
\begin{align}
\cut_k: \ORDER^{\numAGENT} \rightarrow \ORDER^k 
\eqsepv 
\totalordering \mapsto \totalordering_{\vert \{ 1, \ldots, k \} } 
\eqfinv 
\label{eq:cut}
\end{align}
which is the restriction of any (total) ordering of~$\AGENT$ 
to~$\ic{1,k}$.
We define the \emph{set of orderings} by
\begin{equation}
  \ORDER= \bigcup_{ k \in \ic{0,\numAGENT}} \ORDER^k
\quad \text{ where } \ORDER^0 = \{ \emptyset \} 
\eqfinp 
\label{eq:ORDER}
\end{equation}
For any \( k \in \ic{1,\numAGENT} \), and any $k$-ordering~$\kappa \in \ORDER^k$,
we define the \emph{range} $\range{\kappa}$ of the ordering~$\kappa$ as the
subset 
\begin{align}
  \range{\kappa}
&=
\ba{ \kappa(1), \ldots, \kappa(k) }
\subset \AGENT
\eqsepv \forall \kappa \in \ORDER^k
\eqfinv
  \label{range_kappa}
  \intertext{the \emph{cardinal} $\cardinal{\kappa}$ of the ordering $\kappa$ as
  the integer}
\cardinal{\kappa}
&=k
\in \ic{1, \numAGENT}
\eqsepv \forall \kappa \in \ORDER^k
\eqfinv
  \label{cardinal_kappa}
  \intertext{the \emph{last element} $\LastElement{\kappa}$ of the ordering $\kappa$ as the agent}
\LastElement{\kappa}
&=\kappa(k)
\in \AGENT
\eqsepv \forall \kappa \in \ORDER^k
\eqfinv
  \label{LastElement_kappa}
  \intertext{the \emph{restriction} $\FirstElements{\kappa}$ of the ordering $\kappa$ to the first $k{-}1$ elements}
\FirstElements{\kappa}
&= \kappa_{\vert \{ 1, \ldots, k-1 \}} \in \ORDER^{k-1}
\eqsepv \forall \kappa \in \ORDER^k
\eqfinp
  \label{FirstElements_kappa}
\end{align}
\end{subequations}
With the notations introduced,
any ordering $\kappa \in \ORDER \setminus \{ \emptyset \}$
can be written as $\kappa = \np{\FirstElements{\kappa}, \LastElement{\kappa}}$,
with the convention that 
 $\kappa = \np{\LastElement{\kappa}}$ when $\kappa \in \ORDER^1$.

\begin{definition}(\cite{Witsenhausen:1971a,Witsenhausen:1975})
A \emph{configuration-ordering} is a mapping 
$\ordering: \HISTORY \to \ORDER^{\numAGENT}$ from configurations towards total orderings.
With any configuration-ordering~$\ordering$, 
and any ordering $\kappa \in \ORDER$,
we associate the subset \( \HISTORY_{\kappa}^{\ordering} \subset \HISTORY \)
of configurations defined by
\begin{equation}
  \HISTORY_{\kappa}^{\ordering} =
  \defset{\history \in \HISTORY}{\psi_{\cardinal{\kappa}}\bp{\ordering(\history)} =\kappa}  
  \eqsepv \forall \kappa \in \ORDER
  \eqfinp
  \label{eq:HISTORY_k_kappa}
\end{equation}
By convention, we put 
\(   \HISTORY_{\emptyset}^{\ordering} = \HISTORY \). 
\label{de:configuration-ordering}
\end{definition}
Along each configuration $\history \in \HISTORY$, the agents are ordered by
$\ordering(\history) \in \ORDER^{\numAGENT}$.
The set~$\HISTORY_{\kappa}^{\ordering}$ in~\eqref{eq:HISTORY_k_kappa} 
contains all the configurations 
for which the agent~\( \kappa(1) \) is acting first, 
the agent~\( \kappa(2) \) is acting second, \ldots, till 
the last agent~\( \LastElement{\kappa}=\kappa(\cardinal{\kappa}) \) acting at stage~$\cardinal{\kappa}$.

\subsubsection{Causality}

In his article \cite{Witsenhausen:1971a},
Witsenhausen introduces a notion of causality 
and he proves that causal systems are solvable.

The following definition can be interpreted as follows.
In a causal W-model, there exists a configuration-ordering with the following
property: when an agent is called to play --- as he is the last one in an
ordering --- what he knows cannot depend on decisions made by agents
that are not his predecessors 
(in the range of the ordering under consideration).

\begin{definition}(\cite{Witsenhausen:1971a,Witsenhausen:1975})
  \label{de:causality}
  A W-model (as in Definition~\ref{de:W-model}) 
  is \emph{causal} if there exists (at least) one
  configuration-ordering $\ordering: \HISTORY \to \ORDER^{\numAGENT}$ 
  with the property that 
  \begin{equation}
    \label{eq:causality_a}
    \HISTORY_{\kappa}^{\ordering} \cap \History \in 
    \tribu{\History}_{\range{\FirstElements{\kappa}}}
    \eqsepv 
    \forall \History \in \tribu{\Information}_{\LastElement{\kappa}}
    \eqsepv 
    \forall \kappa \in \ORDER 
    \eqfinp
  \end{equation}
\end{definition}
Otherwise said, once we know the first $\cardinal{\kappa}$~agents, 
the information of the (last) agent~$\LastElement{\kappa}$ depends at most
 on the decisions of the (previous) agents in the range~$\range{\FirstElements{\kappa}}$.
In~\eqref{eq:causality_a}, the subset~$\HISTORY_{\kappa}^{\ordering} \subset \HISTORY$ of configurations 
has been defined in~\eqref{eq:HISTORY_k_kappa},
the last agent~$\LastElement{\kappa}$ in~\eqref{LastElement_kappa},
the partial ordering~$\FirstElements{\kappa}$ in~\eqref{FirstElements_kappa},
the range~$\range{\FirstElements{\kappa}}$ in~\eqref{range_kappa},
and --- using the definition~\eqref{eq:sub_history_field_BGENT} 
of the subfield~$\tribu{\History}_{\BGENT}$ of~\( \tribu{\History} \),
with the subset~$\BGENT=\range{\FirstElements{\kappa}}$ of agents 
defined in~\eqref{FirstElements_kappa} and~\eqref{range_kappa} --- 
the subfield $\tribu{\History}_{\range{\FirstElements{\kappa}}}$ 
of~\( \tribu{\History} \) is 
\begin{equation}
      \tribu{\History}_{\range{\FirstElements{\kappa}}}
     =
      \tribu{\NatureField} \otimes 
\bigotimes \limits_{\agent \in \range{\FirstElements{\kappa}}} \tribu{\Control}_{\agent}
      \otimes
      \bigotimes \limits_{\bgent \not\in \range{\FirstElements{\kappa}} } \{
      \emptyset, \CONTROL_\bgent \}
    \subset \tribu{\History}
      \eqfinp
\label{eq:causality_b}
\end{equation}

Witsenhausen's intrinsic model deals with agents, information and strategies,
but not with players and preferences.
We now turn to extending the Witsenhausen's intrinsic model to games.

\section{Finite games in intrinsic form}
\label{Finite_games_in_intrinsic_form}

We are now ready to embed Witsenhausen's intrinsic model
into game theory.
In~\S\ref{Definition_of_a_finite_game_in_intrinsic_form},
we introduce a formal definition of a finite game in intrinsic form (W-game),
and in~\S\ref{Mixed_and_behavioral_strategies}
we introduce three notions of 
``randomization'' of pure strategies --- mixed, product-mixed and behavioral.
In~\S\ref{Strategy_equivalence},
we discuss relations between product-mixed and behavioral W-strategies.

In what follows, when $\SET$ is a finite set, we denote by 
\( \Delta\np{\SET} \) the set of probability distributions over~$\SET$.
When needed, the set~\( \Delta\np{\SET} \) can be equipped with 
the Borel topology and
the Borel $\sigma$-field, as \( \Delta\np{\SET} \)
is homeomorphic to the simplex~$\Sigma_{\cardinal{\SET}}$ of~$\RR^{\cardinal{\SET}}$, 
and is thus homeomorphic to a closed 
subset of a finite dimensional space.

\subsection{Definition of a finite game in intrinsic form (W-game)}
\label{Definition_of_a_finite_game_in_intrinsic_form}

We introduce a formal definition of a finite game in intrinsic form (W-game).

\begin{definition}
  A \emph{finite W-game}
  \( \Bp{ 
    \bp{ 
      \sequence{\AGENT^{\player}}{\player \in \PLAYER}, 
      \np{\Omega, \tribu{\NatureField}},
      \sequence{\CONTROL_{\agent}, \tribu{\Control}_{\agent},
        \tribu{\Information}_{\agent}}{\agent \in \bigcup_{\player \in \PLAYER} \AGENT^{\player}} },
    (\preference^{\player})_{\player \in \PLAYER}
  }
  \),
  or a \emph{finite game in intrinsic form},
  is a made of 
  \begin{itemize}
  \item 
    a family \( \sequence{\AGENT^{\player}}{\player \in \PLAYER} \),
where the set~$\PLAYER$ of \emph{players} is finite,
    of     two by two disjoint nonempty sets     
    whose union $\AGENT= \bigcup_{\player \in \PLAYER} \AGENT^{\player}$ is the set 
    of \emph{agents}; each subset~$\AGENT^{\player}$ is interpreted as 
    the subset of executive agents of the 
    \emph{player} \( \player \in \PLAYER \),
  \item 
    a finite W-model 
    \( \bp{ 
      \AGENT,
      (\Omega, \tribu{\NatureField}), 
      \sequence{\CONTROL_{\agent}, \tribu{\Control}_{\agent},
        \tribu{\Information}_{\agent}}{\agent \in \AGENT} } \),
    as in Definition~\ref{de:W-model},
  \item 
    for each player $\player \in \PLAYER$,
    a preference relation~$\preference^{\player}$ 
    on the set of mappings 
\( \Omega \to \Delta\np{ \prod_{\bgent \in \AGENT} \CONTROL_{\bgent} } \).
  \end{itemize}
  A finite W-game is said to be solvable (resp. causal)
  if the underlying W-model 
  is solvable as in Definition~\ref{de:solvability}
  (resp. causal as in Definition~\ref{de:causality}).
  \label{de:W-game}
\end{definition}

We comment on the preference relations~$\preference^{\player}$
on the set of mappings 
\( \Omega \to \Delta\np{ \prod_{\bgent \in \AGENT} \CONTROL_{\bgent} } \).
Our definition covers
(like in~\cite{Blume-Brandenburger-Dekel:1991})
the most traditional preference relation~$\preference^{\player}$,
which is the numerical \emph{expected utility} preference.
In this latter, each player~$\player \in \PLAYER$
is endowed, on the one hand, with a \emph{criterion} (payoff), that is,
a measurable function $\criterion_\player: (\HISTORY, \tribu{\History}) 
\rightarrow [-\infty,+\infty[$, 
and, on the other hand, with a \emph{belief}, that is,
a probability distribution
 $\probability^\player: \tribu{\NatureField} \rightarrow [0, 1]$
over the states of Nature $(\Omega, \tribu{\NatureField})$.
Then, given \( K_i : \Omega \to \Delta\np{ \prod_{\bgent \in \AGENT}
  \CONTROL_{\bgent} } \), $i=1,2$, one says that
  \( K_1 \preference^{\player} K_2 \) if 
\begin{align*}
  \int_{\Omega} \probability^\player\np{d\omega}
  &
    \int_{ \prod_{\bgent \in \AGENT} \CONTROL_{\bgent} } 
    \criterion_\player\bp{\omega, \sequence{\control_\bgent}{\bgent \in \AGENT}}
    K_1\bp{\omega, d\sequence{\control_\bgent}{\bgent \in \AGENT}}
  \\
  & \leq
  \int_{\Omega} \probability^\player\np{d\omega}
  \int_{ \prod_{\bgent \in \AGENT} \CONTROL_{\bgent} } 
  \criterion_\player\bp{\omega, \sequence{\control_\bgent}{\bgent \in \AGENT}}
  K_2\bp{\omega, d\sequence{\control_\bgent}{\bgent \in \AGENT}}
  \eqfinp
\end{align*}

Note also that the Definition~\ref{de:W-game} includes Bayesian games,
by specifying a product structure for $\Omega$ --- where some factors represent
types of players, and one factor represents chance --- and by considering
additional probability distributions.

\subsection{Mixed, product-mixed and behavioral strategies}
\label{Mixed_and_behavioral_strategies}

We introduce three notions of 
``randomization'' of pure strategies: mixed, product-mixed and behavioral.

The notion of mixed strategy comes from the study of games in normalized form,
where each player  has to select a pure strategy,
the collection of which determines a unique outcome. 
If we allow the players to select their pure strategy at random, 
the lottery they use is called a mixed strategy. 
For an extensive game,  a mixed strategy can be interpreted in the following sense.
First, the player selects a pure strategy using the lottery. 
Second, the game is played. When the player is called by the umpire,
she plays the action specified by the selected pure strategy for the current information set. 

Observe that there is only one dice roll per player. This dice roll determines the reactions of the player for every situation of the game. 
It would be more natural to let the player roll a dice every time she has to
play, leading to the notion of behavioral strategy. 

A fundamental question in game theory is to identify settings in which those two
views (mixed strategy and behavioral strategy) are equivalent.
To formulate this question in the W-game framework, 
we will give formal definitions of these two notions of randomization.
We will also add a third one, that we call
product-mixed strategy, and which 
is in the spirit of Aumann~\cite{Aumann:1964}, 
as each agent (corresponding to time index in~\cite{Aumann:1964}) 
``generates'' strategies from a random device that is independent of all the
other agents.

\subsubsection{Mixed W-strategies}

For any agent~$\agent \in \AGENT$, the set~\( \WSTRATEGY_{\agent}  \) 
of pure W-strategies for agent~$\agent$ (see Definition~\ref{de:W-strategy})
is finite, hence the set \( \Delta\np{\WSTRATEGY_{\agent} } \)
of probability distributions over~\( \WSTRATEGY_{\agent}  \)
is is homeomorphic to 
$\Sigma_{|\WSTRATEGY_{\agent}|}$, the simplex of
$\RR^{|\WSTRATEGY_{\agent}|}$, and is thus homeomorphic to a closed 
subset of a finite dimensional space.
So is the space \( \Delta\np{\WSTRATEGY} \)
of probability distributions over the set~\( \WSTRATEGY \)
of pure W-strategies profiles. 
We will also consider the sets
\begin{equation}
  \WSTRATEGY^{\player} = 
\prod_{\agent \in \AGENT^{\player}} \WSTRATEGY_{\agent}
\eqsepv \forall \player \in \PLAYER
\label{eq:W-strategies_profiles_player}
\end{equation}
of pure W-strategies profiles, player by player, and the set
\( \Delta\np{\WSTRATEGY^{\player}} \) of probability distributions
over~\( \WSTRATEGY^{\player} \).

\begin{subequations}
\begin{definition}
We consider a finite W-game, as in Definition~\ref{de:W-game}.
A \emph{mixed W-strategy} 
for player~$\player \in \PLAYER$ 
is an element~$\MixedStrategy^{\player}$ of \( \Delta\np{\WSTRATEGY^{\player} } \),
the set of probability distributions
over the set~\( \WSTRATEGY^{\player} \) in~\eqref{eq:W-strategies_profiles_player}
of W-strategies of the executive agents in~$\AGENT^{\player}$.
The set of \emph{mixed W-strategies profiles} is
\( \prod_{\player \in \PLAYER} \Delta\np{\WSTRATEGY^{\player}} \). 
%
\label{de:mixed_W-strategy}
\end{definition}
A mixed W-strategies profile is denoted by 
\begin{equation}
\MixedStrategy = \np{\MixedStrategy^{\player}}_{\player \in \PLAYER}
\in \prod_{\player \in \PLAYER} \Delta\bp{ \WSTRATEGY^{\player} }
\eqfinv
\label{eq:mixed_W-strategies_profiles}
\end{equation}
and, when we focus on player~$\player$, we write 
\begin{equation}
\MixedStrategy = 
\couple{{\MixedStrategy}^{-\player}}{{\MixedStrategy}^{\player}} \in
\Delta\bp{ \WSTRATEGY^{\player} } \times
\prod_{\player'\neq \player} \Delta\bp{ \WSTRATEGY^{\player'} }
\eqfinp
\label{eq:mixed_W-strategies_profiles_player}
\end{equation}
\end{subequations}

\begin{subequations}
\begin{definition}
We consider a solvable finite W-game (see Definition~\ref{de:W-game}),
and \( \MixedStrategy = \np{\MixedStrategy^{\player}}_{\player \in \PLAYER}
\in \prod_{\player \in \PLAYER} \Delta\np{ \WSTRATEGY^{\player} } \)
a mixed W-strategies profile as in~\eqref{eq:mixed_W-strategies_profiles}.
For any \( \omega\in\Omega\), we denote by 
\begin{equation}
 \QQ^{\omega}_{\MixedStrategy} =
 \QQ^{\omega}_{\np{\MixedStrategy^{\player}}_{\player \in \PLAYER}} = 
\bp{ \bigotimes_{\player \in \PLAYER}\MixedStrategy^{\player} }
\circ \bp{ \ReducedSolutionMap\np{\omega,\cdot}}^{-1}
\in \Delta\bp{ \prod_{\bgent \in \AGENT} \CONTROL_{\bgent} }
\label{eq:push_forward_probability_a}
\end{equation}
the pushforward probability,
on the space \( \bp{\prod_{\bgent \in \AGENT} \CONTROL_{\bgent},
\bigotimes \limits_{\bgent \in \AGENT} \tribu{\Control}_{\bgent} } \)
of the product probability distribution~\( \bigotimes_{\player \in
  \PLAYER}\MixedStrategy^{\player} \)
on~\( \prod_{\player \in \PLAYER} \WSTRATEGY^{\player} = \WSTRATEGY \)
by the mapping
\begin{equation}
  \ReducedSolutionMap\np{\omega,\cdot} : \WSTRATEGY \to 
\prod_{\bgent \in \AGENT} \CONTROL_{\bgent} 
 \eqsepv 
\wstrategy \mapsto \ReducedSolutionMap_{\wstrategy}(\omega) 
\eqfinv
\label{eq:push_forward_probability_b}
\end{equation}
where \( \ReducedSolutionMap_{\wstrategy} \) is the 
solution map~\eqref{eq:solution_map}, 
which exists by the solvability assumption.
\end{definition}
\label{eq:push_forward_probability}
\end{subequations}

By~\eqref{eq:solution_map_IFF}, which defines the solution map,
and by definition of a pushforward probability,
we have, for any configuration~\( \bp{\omega, 
\sequence{\control_\bgent}{\bgent\in\AGENT} } \in\HISTORY \),
\begin{align*}
  \QQ^{\omega}_{\MixedStrategy}
  \bp{\sequence{\control_\bgent}{\bgent\in\AGENT} }
  & =
    \bp{\bigotimes_{\player \in \PLAYER}\MixedStrategy^{\player}}
    \Bp{ \ReducedSolutionMap\np{\omega,\cdot}^{-1}
    \bp{\sequence{\control_\bgent}{\bgent\in\AGENT} } } 
  \\
  &=
    \prod_{\player \in \PLAYER} \MixedStrategy^{\player}
    \bgp{\Bset{ \sequence{\wstrategy_{\agent}}{\agent\in\AGENT^{\player} }
    \in\WSTRATEGY^{\player} }%
    { \lambda_{\agent}\bp{\omega, 
    \sequence{\control_\bgent}{\bgent\in\AGENT} } =\control_{\agent}
    \eqsepv \forall \agent\in \AGENT^{\player} }}
    \eqfinp
\end{align*}

\subsubsection{Product-mixed W-strategies}

In a mixed W-strategy, the executive agents of player~$\player \in \PLAYER$ 
can be correlated because the probability~$\MixedStrategy^{\player}$
in Definition~\ref{de:mixed_W-strategy}
is a joint probability on the product space
\( \WSTRATEGY^{\player} =\prod_{\agent \in \AGENT^{\player}} \WSTRATEGY_{\agent}
\).
We now introduce product-mixed W-strategies, where 
the executive agents of player~$\player \in \PLAYER$ are independent
in the sense that the probability~$\MixedStrategy^{\player}$ is  the
product of individual probabilities,
each of them on the individual space~\( \WSTRATEGY_{\agent} \)
of the strategies of one agent~$\agent$. 

\begin{definition}
  \label{de:product-mixed_W-strategy}
  We consider a finite W-game, as in Definition~\ref{de:W-game}.
  A \emph{product-mixed W-strategy} for player~$\player \in \PLAYER$ is
  an element 
  \( \ProductMixedStrategy^{\player}=
  \sequence{ \ProductMixedStrategy^{\player}_{\agent}}{\agent \in \AGENT^{\player}} \) of 
  \( \prod_{\agent \in \AGENT^{\player}} \Delta\np{\WSTRATEGY_{\agent}} \).
\end{definition}
The product-mixed W-strategy 
\( \sequence{ \ProductMixedStrategy^{\player}_{\agent}}{\agent \in \AGENT^{\player}} \)
induces a product probability\footnote{%
By an abuse of notation, we will sometimes write 
\(  \ProductMixedStrategy^{\player}=
\otimes_{\agent\in\AGENT^{\player}}\ProductMixedStrategy^{\player}_{\agent} \).
\label{ft:product-mixed_W-strategy}}
\( 
\otimes_{\agent\in\AGENT^{\player}}\ProductMixedStrategy^{\player}_{\agent} \)
on the set~\( \WSTRATEGY^{\player} \),
which is a mixed W-strategy
as in Definition~\ref{de:mixed_W-strategy}.

\subsubsection{Behavioral W-strategies}

We formalize the intuition of behavioral strategies in W-games
by the following definition  of behavioral W-strategies.

\begin{definition}
We consider a finite W-game, as in Definition~\ref{de:W-game}.
A \emph{behavioral W-strategy}
for player~$\player \in \PLAYER$ 
is a family 
\( \BehavioralWStrategy^{\player} = 
\sequence{\BehavioralWStrategy^{\player}_{\agent}}{\agent \in \AGENT^{\player}} \),
where 
\begin{equation}
\BehavioralWStrategy^{\player}_{\agent} :
\HISTORY \times \tribu{\Control}_{\agent} \to [0,1] 
\eqsepv 
\np{\history, \Control_{\agent}} \mapsto  \BehavioralWStrategy^{\player}_{\agent}
\conditionalySet{\Control_{\agent}}{\history} 
\end{equation}
is an \( \tribu{\Information}_{\agent} \)-measurable stochastic kernel
for each \( \agent \in \AGENT^{\player} \), that is,
if one of the two equivalent statements holds true:
\begin{enumerate}
\item 
\label{it:behavioral_W-strategy_abstract}
on the one hand, 
the function
\( \history \in \HISTORY \mapsto \BehavioralWStrategy^{\player}_{\agent}
\conditionaly{\control_{\agent}}{\history} \) 
is \( \tribu{\Information}_{\agent} \)-measurable,
for any \( \control_{\agent}\in\CONTROL_{\agent} \) and,
on the other hand, 
each \( \BehavioralWStrategy^{\player}_{\agent}
\conditionalySet{ \cdot }{ \history } \) is 
a probability distribution on the finite set~\( \CONTROL_{\agent} \),
for any \( \history \in \HISTORY \),
\item 
\label{it:behavioral_W-strategy}
on the one hand, 
\( \history' \sim_{\agent} \history'' \implies
\BehavioralWStrategy^{\player}_{\agent}
\conditionaly{\control_{\agent}}{\history'}
=
\BehavioralWStrategy^{\player}_{\agent}
\conditionaly{\control_{\agent}}{\history''}
 \), for any \( \control_{\agent}\in\CONTROL_{\agent} \)
and, on the other hand,
for any \( \history \in \HISTORY \),
we have \( \BehavioralWStrategy^{\player}_{\agent}
\conditionaly{\control_{\agent}}{\history} \geq 0 \),
\( \forall \control_{\agent}\in\CONTROL_{\agent} \),
and
\(  \sum_{ \control_{\agent}\in\CONTROL_{\agent} } 
\BehavioralWStrategy^{\player}_{\agent}
\conditionaly{\control_{\agent}}{\history} = 1 \).
\end{enumerate}
\label{de:behavioral_W-strategy}
\end{definition}
The equivalences come from the fact that the sets \( \HISTORY \) and \( \CONTROL_{\agent} \)
are finite and equipped with their respective complete fields, and
by Proposition~\ref{pr:measurability}, 
and especially~\eqref{eq:strategy_atoms_rho}.

\subsection{Relations between product-mixed and behavioral W-strategies}
\label{Strategy_equivalence}

Here, we show that product-mixed and behavioral W-strategies
are ``equivalent'' in the sense that
a product-mixed W-strategy naturally induces a 
behavioral W-strategy, and that a behavioral W-strategy can be ``realized''
as a product-mixed W-strategy
(see Figure~\ref{fig:randomization}).

\subsubsection*{From product-mixed to behavioral W-strategies}

We prove that a product-mixed W-strategy naturally induces a 
behavioral W-strategy.

\begin{proposition}
  \label{pr:PMtoB}
  We consider a finite W-game, as in Definition~\ref{de:W-game},
  and a player \( \player \in \PLAYER \).

  For any product-mixed W-strategy
  \( \ProductMixedStrategy^{\player}=
  \sequence{ \ProductMixedStrategy^{\player}_{\agent}}{\agent \in
    \AGENT^{\player}} \in \prod_{\agent \in \AGENT^{\player}} \Delta\np{\WSTRATEGY_{\agent}} \), 
  as in Definition~\ref{de:product-mixed_W-strategy}, 
  we define, for any agent~\( \agent\in\AGENT^{\player} \), 
  \begin{equation}
    \PMtoB{\ProductMixedStrategy}^{\player}_{\agent}
    \conditionaly{\control_{\agent}}{\history}
    =
    \ProductMixedStrategy^{\player}_{\agent}
    \Bp{
    \defset{ \wstrategy_{\agent}\in\WSTRATEGY_{\agent}}%
    {\wstrategy_{\agent}\np{\history}=\control_{\agent}}}
    \eqsepv \forall \control_{\agent}\in\CONTROL_{\agent} 
    \eqsepv \forall \history \in \HISTORY 
    \eqfinp
    \label{eq:PMtoB}
  \end{equation}
  Then, \( \PMtoB{\ProductMixedStrategy}^{\player}=
  \sequence{ \PMtoB{\ProductMixedStrategy}^{\player}_{\agent}}{\agent \in
    \AGENT^{\player}} \) is a behavioral W-strategy, 
  as in Definition~\ref{de:behavioral_W-strategy}.
\end{proposition}

\begin{proof}
Let be given a product-mixed W-strategy
\( \ProductMixedStrategy^{\player}=
\sequence{ \ProductMixedStrategy^{\player}_{\agent}}{\agent \in
  \AGENT^{\player}} \in \prod_{\agent \in \AGENT^{\player}} \Delta\np{\WSTRATEGY_{\agent}} \). 

To prove that~\eqref{eq:PMtoB} defines a 
behavioral W-strategy, 
we have to show (see Item~\ref{it:behavioral_W-strategy_abstract}
in Definition~\ref{de:behavioral_W-strategy}), 
on the one hand, that the function
\( \history \in \HISTORY \mapsto \PMtoB{\ProductMixedStrategy}^{\player}_{\agent}
\conditionaly{\control_{\agent}}{\history} \) 
is \( \tribu{\Information}_{\agent} \)-measurable,
for any \( \control_{\agent}\in\CONTROL_{\agent} \) and,
on the other hand, that 
each \( \BehavioralWStrategy^{\player}_{\agent}
\conditionalySet{ \cdot }{ \history } \) is 
a probability on the finite set \( \CONTROL_{\agent} \),
for any \( \history \in \HISTORY \).
For this purpose, we will use the more practical characterization of 
Item~\ref{it:behavioral_W-strategy}
in Definition~\ref{de:behavioral_W-strategy}.

Let us fix \( \control_{\agent}\in\CONTROL_{\agent} \).
Let \( \history', \history'' \in \HISTORY \) be such that
\( \history' \sim_{\agent} \history'' \),
where we recall that the classes of the
equivalence relation~$\sim_{\agent}$ in~\eqref{eq:bracket_HISTORY}
are exactly the atoms in~\( \crochet{\tribu{\Information}_{\agent}} \).
By~\eqref{eq:common_value_equivalence_class_rho}, we have that 
\( \defset{ \wstrategy_{\agent}\in\WSTRATEGY_{\agent}}%
{\wstrategy_{\agent}\np{\history'}=\control_{\agent}}
=
\defset{ \wstrategy_{\agent}\in\WSTRATEGY_{\agent}}%
{\wstrategy_{\agent}\np{\history''}=\control_{\agent}}
\). 
Therefore, from the expression~\eqref{eq:PMtoB},
we have obtained that 
\( \history' \sim_{\agent} \history'' \implies
\PMtoB{\ProductMixedStrategy}^{\player}_{\agent}
\conditionaly{\control_{\agent}}{\history'}
=
\PMtoB{\ProductMixedStrategy}^{\player}_{\agent}
\conditionaly{\control_{\agent}}{\history''}
 \), hence that the function
\( \history \in \HISTORY \mapsto \PMtoB{\ProductMixedStrategy}^{\player}_{\agent}
\conditionaly{\control_{\agent}}{\history} \) 
is \( \tribu{\Information}_{\agent} \)-measurable,
by Proposition~\ref{pr:measurability}, 
and especially~\eqref{eq:strategy_atoms_rho}. 

By the expression~\eqref{eq:PMtoB},
we have that 
\( \PMtoB{\ProductMixedStrategy}^{\player}_{\agent}
\conditionaly{\control_{\agent}}{\history} \geq 0 \),
\( \forall \control_{\agent}\in\CONTROL_{\agent} \)
and, since \( \ProductMixedStrategy^{\player}_{\agent} \)
is a probability on~\( \Delta\np{\WSTRATEGY_{\agent}} \),
that 
\(  \sum_{ \control_{\agent}\in\CONTROL_{\agent} } 
\PMtoB{\ProductMixedStrategy}^{\player}_{\agent}
\conditionaly{\control_{\agent}}{\history} = 1 \).
\medskip

This ends the proof.
\end{proof}

\subsubsection*{From behavioral to product-mixed W-strategies}

We prove that a behavioral W-strategy can be ``realized''
as a product-mixed W-strategy.

\begin{proposition}
  \label{pr:BtoPM} 
  We consider a finite W-game, as in Definition~\ref{de:W-game},
  and a player \( \player \in \PLAYER \).

  For any behavioral W-strategy \( \BehavioralWStrategy^{\player} = 
  \sequence{\BehavioralWStrategy^{\player}_{\agent}}{\agent \in \AGENT^{\player}}
  \), as in Definition~\ref{de:behavioral_W-strategy},
  there exists a product-mixed W-strategy
  \( \BtoPM{\BehavioralWStrategy}^{\player} = 
  \sequence{\BtoPM{\BehavioralWStrategy}^{\player}_{\agent}}{\agent \in \AGENT^{\player}}\),
  as in Definition~\ref{de:product-mixed_W-strategy}, 
  with the property that, for any agent~\( \agent \) in \(\AGENT^{\player}\) we have
  \begin{equation}
    \BtoPM{\BehavioralWStrategy}^{\player}_{\agent}
    \Bp{\defset{ \wstrategy_{\agent}\in\WSTRATEGY_{\agent}}%
      {\wstrategy_{\agent}\np{\history}=\control_{\agent}}}
    = 
    \BehavioralWStrategy^{\player}_{\agent}
    \conditionaly{\control_{\agent}}{\history}
    \eqsepv \forall \control_{\agent}\in\CONTROL_{\agent} 
    \eqsepv \forall \history \in \HISTORY 
    \eqfinp
    \label{eq:BtoPM}
  \end{equation}
\end{proposition}
\begin{subequations}
  \begin{proof} 
    We consider a fixed agent~\( \agent\in\AGENT^{\player} \).

    On the one hand,
    by Proposition~\ref{pr:measurability}, we get that
    \begin{equation*}
      \wstrategy_{\agent} \in \WSTRATEGY_{\agent} \iff
      \exists \sequence{\control_{\agent}^{\atom_{\agent}}}{%
        \atom_{\agent} \in \crochet{\tribu{\Information}_{\agent}}} 
      \in 
      \CONTROL_{\agent}^{\crochet{\tribu{\Information}_{\agent}}}
      \eqsepv 
      \wstrategy_{\agent}\np{\history}=\control_{\agent}^{\atom_{\agent}}
      \eqsepv \forall \atom_{\agent} \in \crochet{\tribu{\Information}_{\agent}}
      \eqsepv \forall \history \in \atom_{\agent}
      \eqfinv
    \end{equation*}
where \( \CONTROL_{\agent}^{\crochet{\tribu{\Information}_{\agent}}} \) is the
set of mappings from~\( \crochet{\tribu{\Information}_{\agent}} \) to~\(
\CONTROL_{\agent} \).
Therefore, the following mapping is a bijection:
    \begin{equation}
      \Psi : \WSTRATEGY_{\agent} \to 
      \CONTROL_{\agent}^{\crochet{\tribu{\Information}_{\agent}}}
      \eqsepv
      \wstrategy_{\agent} \mapsto 
      \sequence{\wstrategy_{\agent}\np{\atom_{\agent}}}{%
        \atom_{\agent} \in \crochet{\tribu{\Information}_{\agent}}} 
      \eqfinp
      \label{eq:BtoPM_proof_bijection} 
    \end{equation}
    We denote the inverse bijection by
    \begin{equation}
      \Phi=\Psi^{-1} : 
      \CONTROL_{\agent}^{\crochet{\tribu{\Information}_{\agent}}}
      \to \WSTRATEGY_{\agent} 
      \eqfinp
      \label{eq:BtoPM_proof_bijection_inverse} 
    \end{equation}

    On the other hand, 
    by Item~\ref{it:behavioral_W-strategy_abstract}
    in Definition~\ref{de:behavioral_W-strategy}, each 
    \( \BehavioralWStrategy^{\player}_{\agent}
    \conditionalySet{ \cdot }{ \history } \) is 
    a probability on the finite set \( \CONTROL_{\agent} \),
    for any \( \history \in \HISTORY \).
    As the mapping \( \history \in \HISTORY \mapsto 
    \BehavioralWStrategy^{\player}_{\agent}
    \conditionalySet{ \cdot }{ \history } \) is 
    \( \tribu{\Information}_{\agent} \)-measurable,
    by Definition~\ref{de:behavioral_W-strategy}, 
    the notation \( \BehavioralWStrategy^{\player}_{\agent}
    \conditionalySet{ \cdot }{ \atom_{\agent} } \) 
    makes sense by~\eqref{eq:common_value_rho}.

    We equip the finite set 
    \( \CONTROL_{\agent}^{\crochet{\tribu{\Information}_{\agent}}} \)
    with the product probability
    \( \bigotimes_{ \atom_{\agent} \in \crochet{\tribu{\Information}_{\agent}} } 
    \BehavioralWStrategy^{\player}_{\agent}
    \conditionalySet{ \cdot }{ \atom_{\agent} } \),
    and we define the pushforward probability
    \begin{equation}
      \BtoPM{\BehavioralWStrategy}^{\player}_{\agent}
      = \Bp{ \bigotimes_{ \atom_{\agent} \in \crochet{\tribu{\Information}_{\agent}} } 
        \BehavioralWStrategy^{\player}_{\agent}
        \conditionalySet{ \cdot }{ \atom_{\agent} } }
      \circ \Phi^{-1} \eqsepv
      \label{eq:BtoPM_proof_probability}
    \end{equation}
    on the finite set~\( \WSTRATEGY_{\agent} \).
    Then, we calculate, for any \( \history \in \HISTORY \), 
    \begin{align*}
      \BtoPM{\BehavioralWStrategy}^{\player}_{\agent}
      &
        \Bp{\defset{ \wstrategy_{\agent}\in\WSTRATEGY_{\agent}}%
        {\wstrategy_{\agent}\np{\history}=\control_{\agent}}}
      \\
      &=  
          \Bp{ \bigotimes_{ \atom_{\agent} \in \crochet{\tribu{\Information}_{\agent}} } 
          \BehavioralWStrategy^{\player}_{\agent}
          \conditionalySet{ \cdot }{ \atom_{\agent} } }
          \Bp{ \Phi^{-1} \bp{ \defset{ \wstrategy_{\agent}\in\WSTRATEGY_{\agent}}%
          {\wstrategy_{\agent}\np{\history}=\control_{\agent}} }}
          \tag{by definition~\eqref{eq:BtoPM_proof_probability}}
      \\
      &=
          \Bp{ \bigotimes_{ \atom_{\agent} \in \crochet{\tribu{\Information}_{\agent}} } 
          \BehavioralWStrategy^{\player}_{\agent}
          \conditionalySet{ \cdot }{ \atom_{\agent} } }
          \Bp{ \Psi \bp{\defset{ \wstrategy_{\agent}\in\WSTRATEGY_{\agent}}%
          {\wstrategy_{\agent}\np{\history}=\control_{\agent}} }}
          \tag{as \( \Phi^{-1}=\Psi \) by~\eqref{eq:BtoPM_proof_bijection_inverse}}
      \\
      &=
          \Bp{ \bigotimes_{ \atom_{\agent} \in \crochet{\tribu{\Information}_{\agent}} } 
          \BehavioralWStrategy^{\player}_{\agent}
          \conditionalySet{ \cdot }{ \atom_{\agent} } }
          \Bp{ \{ \control_{\agent} \} \times
          \CONTROL_{\agent}^{ \crochet{\tribu{\Information}_{\agent}}\backslash \bracket{\history}_{\agent} }
          }
          \intertext{because, by definition of the mapping~$\Psi$
          in~\eqref{eq:BtoPM_proof_bijection},
          any \( \wstrategy_{\agent} \in
          \defset{ \wstrategy'_{\agent}\in\WSTRATEGY_{\agent}}%
          {\wstrategy'_{\agent}\np{\history}=\control_{\agent}} \) 
          takes the value~$\control_{\agent}$ 
          on the atom~\( \bracket{\history}_{\agent} \)
          (by definition of the set~\( \WSTRATEGY_{\agent} \)
          in Definition~\ref{de:W-strategy} and by~\eqref{eq:common_value_equivalence_class_rho}),
          and any possible value in~\( \CONTROL_{\agent} \) for all the 
          other atoms in \( \crochet{\tribu{\Information}_{\agent}}
          \backslash \bracket{\history}_{\agent} \) }
      &=  
          \BehavioralWStrategy^{\player}_{\agent}
          \conditionaly{ \control_{\agent} }{ \bracket{\history}_{\agent} }
          \times
          \prod_{ \atom_{\agent} \in \crochet{\tribu{\Information}_{\agent}}
          \backslash \bracket{\history}_{\agent} } 
          \BehavioralWStrategy^{\player}_{\agent}
          \conditionalySet{ \CONTROL_{\agent} }{ \atom_{\agent} }
          \tag{by definition of the product probability
          \( \bigotimes_{ \atom_{\agent} \in \crochet{\tribu{\Information}_{\agent}} } 
          \BehavioralWStrategy^{\player}_{\agent}
          \conditionalySet{ \cdot }{ \atom_{\agent} } \) }
      \\
      &=  
          \BehavioralWStrategy^{\player}_{\agent}
          \conditionaly{ \control_{\agent} }{ \bracket{\history}_{\agent} }
          \times
          \prod_{ \atom_{\agent} \in \crochet{\tribu{\Information}_{\agent}}
          \backslash \bracket{\history}_{\agent} } 
          1
          \tag{as \( \BehavioralWStrategy^{\player}_{\agent}
          \conditionalySet{ \CONTROL_{\agent} }{ \atom_{\agent} }=1 \)
          for all \( \atom_{\agent} \in \crochet{\tribu{\Information}_{\agent}} \)}
      \\
      &=  
          \BehavioralWStrategy^{\player}_{\agent}
          \conditionaly{ \control_{\agent} }{ \bracket{\history}_{\agent} }
      \\
      &=  
          \BehavioralWStrategy^{\player}_{\agent}
          \conditionaly{\control_{\agent}}{\history}
    \end{align*}
    as the function \( \history \in \HISTORY \mapsto 
    \BehavioralWStrategy^{\player}_{\agent}
    \conditionaly{\control_{\agent}}{\history} \) is 
    \( \tribu{\Information}_{\agent} \)-measurable,
    by Item~\ref{it:behavioral_W-strategy_abstract}
    in Definition~\ref{de:behavioral_W-strategy}, 
    and using~\eqref{eq:common_value_equivalence_class_rho}.
    \medskip

    This ends the proof. 
  \end{proof}
\end{subequations}

\section{Kuhn's equivalence theorem}
\label{Kuhn_Theorem}

Now, we are equipped to give, for games in intrinsic form, a statement and a proof of the 
celebrated Kuhn's equivalence theorem: when a player enjoys perfect recall, 
for any mixed W-strategy, there is an equivalent behavioral strategy.

\begin{figure}
  \centering
  \[
  \begin{tikzcd}[column sep="2cm",row sep="4cm",cells={nodes={draw=gray}},labels={font=\everymath\expandafter{\the\everymath\textstyle}}]
    & \substack{\textrm{\normalsize mixed} \\ \textrm{\normalsize W-strategies}}
    \arrow{dr}[description]{{\textrm{Proposition~\ref{pr:MtoB_PerfectRecall} under perfect recall}}}
    &
    \\
    \substack{\textrm{\normalsize product mixed} \\ \textrm{\normalsize W-strategies}}
    \arrow[ru,bend left=10,"\textrm{injection}" ]
    \arrow[rr,bend left=10,"\textrm{Proposition~\ref{pr:PMtoB}}"]
    &
    &
    \substack{\textrm{\normalsize behavioral} \\ \textrm{\normalsize W-strategies}}
    \arrow[ll,bend left=10,"\textrm{Proposition~\ref{pr:BtoPM}}"] 
  \end{tikzcd}
  \]
  \caption{\label{fig:randomization}Three Propositions that relate three notions of
    randomization of strategies}
\end{figure}
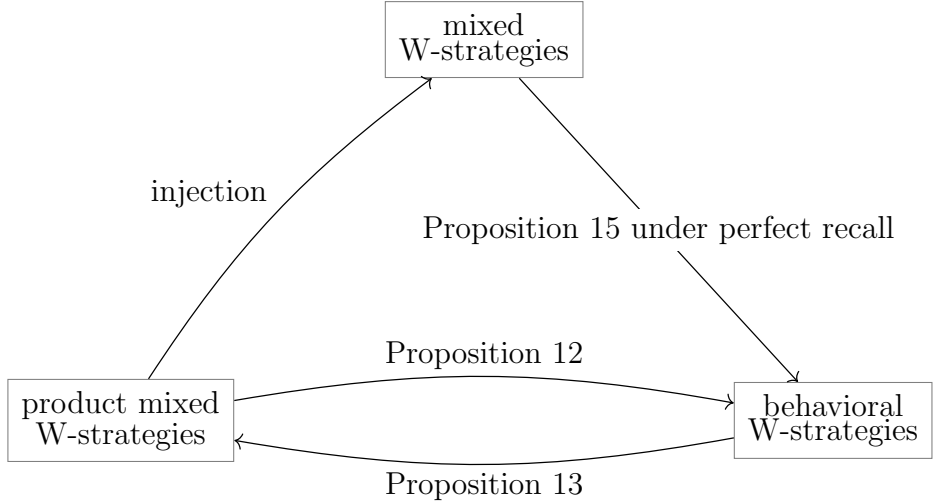

In this Section, we consider a causal finite W-game
(see Definition~\ref{de:W-game}), 
that is, the underlying W-model (as in Definition~\ref{de:W-model}) 
is causal (see Definition~\ref{de:causality}), 
with suitable configuration-ordering $\ordering: \HISTORY \to \ORDER^{\numAGENT}$.
In~\S\ref{Perfect_recall}, we introduce a formal definition of perfect recall 
in a causal finite game in intrinsic form. 
In~\S\ref{Preliminary_results} (Proposition~\ref{pr:MtoB_PerfectRecall}), we show that 
any mixed W-strategy induces a behavioral W-strategy under perfect recall
(see Figure~\ref{fig:randomization}).
Finally, in~\S\ref{Main_result} (Theorem~\ref{th:Kuhn}), we give a statement and
a proof of Kuhn's equivalence theorem for games in intrinsic form.

\subsection{Definition of perfect recall for causal W-games}
\label{Perfect_recall}

For any agent~\( \agent\in\AGENT \), 
we define the \emph{choice field}~\( \tribu{\Choice}_{\agent} 
\subset \tribu{\History} \) by
   \begin{equation}
\tribu{\Choice}_{\agent} = \tribu{\Control}_{\agent} \bigvee 
\tribu{\Information}_{\agent} \eqsepv \forall \agent \in \AGENT
\eqfinp
   \label{eq:ChoiceField}
   \end{equation}
Thus defined, the choice field of an agent contains both what the agent did
and what he knew when making the decision.

The following definition of perfect recall is new.
\begin{subequations}
\begin{definition}
We consider a causal finite W-game for which the underlying W-model 
is causal with the configuration-ordering $\ordering: \HISTORY \to \ORDER^{\numAGENT}$.

 We say that a player \( \player \in \PLAYER \) enjoys \emph{perfect recall} 
if, for any ordering \( \kappa \in \ORDER \)
such that~\( \LastElement{\kappa} \in \AGENT^{\player} \)
(that is, the last agent is an executive of the player), we have 
  \begin{equation}
\HISTORY_{\kappa}^{\ordering} \cap \History \in 
\tribu{\Information}_{\LastElement{\kappa}}
\eqsepv 
\forall \History \in \tribu{\Choice}_{\range{\FirstElements{\kappa}} \cap
  \AGENT^{\player}} 
\eqfinv    
   \label{eq:PerfectRecall}
  \end{equation}
where the subset~$\HISTORY_{\kappa}^{\ordering} \subset \HISTORY$ of configurations 
has been defined in~\eqref{eq:HISTORY_k_kappa},
the last agent~$\LastElement{\kappa}$ in~\eqref{LastElement_kappa},
the partial ordering~$\FirstElements{\kappa}$ in~\eqref{FirstElements_kappa},
the range~$\range{\FirstElements{\kappa}}$ in~\eqref{range_kappa},
and where
\begin{equation}
\tribu{\Choice}_{\range{\FirstElements{\kappa}} \cap \AGENT^{\player}} =
\bigvee \limits_{\substack{\agent \in \range{\FirstElements{\kappa}}\\
\agent \in \AGENT^{\player}}}
\tribu{\Choice}_{\agent} 
\eqfinv
   \label{eq:PastChoiceField}
\end{equation}
with the choice subfield \( \tribu{\Choice}_{\agent} \subset \tribu{\History} \) given by~\eqref{eq:ChoiceField}.
   \label{de:PerfectRecall}
\end{definition}
\end{subequations}
We interpret the above definition 
as follows.
A player enjoys perfect recall when any of her executive
agents --- when called to play as the last one in an ordering --- 
knows at least what did and knew those of the 
executive agents that are both his predecessors 
(in the range of the ordering under consideration)
and that are executive agents of the player.

\subsection{Mixed W-strategy induces behavioral W-strategy under perfect recall}
\label{Preliminary_results}

As a preparatory result for the proof of Kuhn's equivalence theorem, we show that 
any mixed W-strategy induces a behavioral W-strategy under perfect recall.

\begin{proposition}
We consider a causal finite W-game
(see Definition~\ref{de:W-game}),
for which the underlying W-model (see Definition~\ref{de:W-model}) 
is causal (see Definition~\ref{de:causality})
with the configuration-ordering $\ordering: \HISTORY \to \ORDER^{\numAGENT}$.

We consider a player \( \player \in \PLAYER \),
equipped with a mixed W-strategy \( \MixedStrategy^{\player} \in
\Delta\np{\WSTRATEGY^{\player}} \) and 
supposed to enjoy perfect recall (see Definition~\ref{de:PerfectRecall}),
and an agent~\( \agent\in\AGENT^{\player} \).

Then, for each agent \( \agent \in \AGENT^{\player} \), 
the following formula\footnote{%
With the convention that 
\( \MtoB{\MixedStrategy}^{\player}_{\agent}\np{ \{\control_{\agent}\} \mid
  \history}=0 \) if the denominator
is zero (in which case the numerator, which is smaller, is also zero),
and using the notations
\( \wstrategy=\sequence{\wstrategy_{\agent}}{\agent \in \AGENT^{\player}}
\in\WSTRATEGY^{\player} \), 
\eqref{eq:sub_history_BGENT} for \( \history_\BGENT \)
and~\eqref{eq:sub_wstrategy_BGENT} for \( \wstrategy_\BGENT \), 
with \( \BGENT = \range{\FirstElements{\kappa}} \cap \AGENT^{\player} \),
where
the last agent~$\LastElement{\kappa}$  has been defined in~\eqref{LastElement_kappa},
the partial ordering~$\FirstElements{\kappa}$ in~\eqref{FirstElements_kappa}
and the range~$\range{\FirstElements{\kappa}}$ in~\eqref{range_kappa}.
} 
\begin{equation}
  \begin{split}
  \MtoB{\MixedStrategy}^{\player}_{\agent}\conditionaly{\control_{\agent}}{\history}
= \\
\frac{ \MixedStrategy^{\player} \defset{ \wstrategy\in\WSTRATEGY^{\player} }{ 
\exists \kappa \in \ORDER,
\, \history\in\HISTORY_{\kappa}^{\ordering} ,
\, \LastElement{\kappa}=\agent,
\, \wstrategy_{\agent}\np{\history}=\control_{\agent},
\, \wstrategy_{ \range{\FirstElements{\kappa}} \cap \AGENT^{\player} }\np{\history} 
= \history_{\range{\FirstElements{\kappa}} \cap \AGENT^{\player}} } }%
{ \MixedStrategy^{\player} \defset{ \wstrategy\in\WSTRATEGY^{\player} }{ \exists \kappa \in \ORDER,
\, \history\in\HISTORY_{\kappa}^{\ordering} ,
\, \LastElement{\kappa}=\agent,
\, \wstrategy_{\range{\FirstElements{\kappa}} \cap \AGENT^{\player}}\np{\history} 
= \history_{\range{\FirstElements{\kappa}} \cap \AGENT^{\player}} } }    
  \end{split}
\label{eq:MtoB_PerfectRecall}
\end{equation}
defines an \( \tribu{\Information}_{\agent} \)-measurable stochastic kernel
\( \MtoB{\MixedStrategy}^{\player}_{\agent}\).
As a consequence, the family 
\( \MtoB{\MixedStrategy}^{\player} = 
\sequence{\MtoB{\MixedStrategy}^{\player}_{\agent}}{\agent \in \AGENT^{\player}} \),
is a behavioral W-strategy, as in Definition~\ref{de:behavioral_W-strategy}.
\label{pr:MtoB_PerfectRecall}
\end{proposition}

\begin{subequations}
\begin{proof}
We consider a player \( \player \in \PLAYER \),
equipped with a mixed W-strategy \( \MixedStrategy^{\player} \in
\Delta\np{\WSTRATEGY^{\player}} \) and 
supposed to enjoy perfect recall as in Definition~\ref{de:PerfectRecall},
and an agent~\( \agent\in\AGENT^{\player} \).
\medskip

By~\eqref{eq:MtoB_PerfectRecall}, it is easy to see that, 
for any \( \history \in \HISTORY \),
we have \( \MtoB{\MixedStrategy}^{\player}_{\agent}
\conditionaly{\control_{\agent}}{\history} \geq 0 \),
\( \forall \control_{\agent}\in\CONTROL_{\agent} \),
and
\(  \sum_{ \control_{\agent}\in\CONTROL_{\agent} } 
\MtoB{\MixedStrategy}^{\player}_{\agent}
\conditionaly{\control_{\agent}}{\history} = 1 \).
Therefore, by Item~\ref{it:behavioral_W-strategy_abstract}
in Definition~\ref{de:behavioral_W-strategy}, 
there remains to prove that 
the function \( \history \in \HISTORY \mapsto 
\MtoB{\MixedStrategy}^{\player}_{\agent}
\conditionaly{\control_{\agent}}{\history} \) is 
\( \tribu{\Information}_{\agent} \)-measurable.
The proof is in several steps.
\medskip

\noindent$\bullet$
As the agent~\( \agent\in\AGENT \) is fixed,
it is easily seen that the family \( \sequence{ \HISTORY_{\kappa}^{\ordering} }{
\kappa \in \ORDER, \, \LastElement{\kappa}=\agent } \),
where the subset~$\HISTORY_{\kappa}^{\ordering} \subset \HISTORY$ of configurations 
has been defined in~\eqref{eq:HISTORY_k_kappa}, is made of
(possibly empty) disjoint sets whose union is~\( \HISTORY \).
Indeed, on the one hand, for any \( \history\in \HISTORY \), we have that 
 \( \history\in \HISTORY_{\kappa_0}^{\ordering} \) where 
\( \kappa_0=\psi_k\np{\totalordering} \), with
\( \totalordering=\varphi\np{\history} \) 
and $k$ the unique integer such that \( \totalordering(k)=\agent \).
On the other hand, if we had \( \HISTORY_{\kappa_1}^{\ordering}
\cap \HISTORY_{\kappa_2}^{\ordering} \neq \emptyset \),
with \( \LastElement{\kappa_1}=\LastElement{\kappa_2}=\agent \), 
then \( \history\in \HISTORY_{\kappa_1}^{\ordering}
\cap \HISTORY_{\kappa_2}^{\ordering} \) would be such that
\( \kappa_1=\kappa_2=\psi_k\np{\totalordering} \), with
\( \totalordering=\varphi\np{\history} \) 
and $k$ the unique integer such that \( \totalordering(k)=\agent \).
Thus, \( \kappa_1 \neq \kappa_2 \implies \HISTORY_{\kappa_1}^{\ordering}
\cap \HISTORY_{\kappa_2}^{\ordering} = \emptyset \).
\medskip

\noindent$\bullet$
As the family \( \sequence{ \HISTORY_{\kappa}^{\ordering} }{
    \kappa \in \ORDER, \, \LastElement{\kappa}=\agent } \) is made of
    disjoint sets whose union is~\( \HISTORY \), 
we rewrite~\eqref{eq:MtoB_PerfectRecall} as
\begin{align}
  \MtoB{\MixedStrategy}^{\player}_{\agent}
  &\conditionaly{\control_{\agent}}{\history}
    \nonumber
  \\
  &=\frac{{\displaystyle \mathop{\sum}_{\substack{ \kappa \in \ORDER\\\LastElement{\kappa}=\agent}}}
    \MixedStrategy^{\player} \defset{ \wstrategy\in\WSTRATEGY^{\player} }{ 
    \history\in\HISTORY_{\kappa}^{\ordering} 
    \eqsepv \wstrategy_{\agent}\np{\history}=\control_{\agent}
    \eqsepv \wstrategy_{\range{\FirstElements{\kappa}} \cap \AGENT^{\player}}\np{\history} 
    = \history_{\range{\FirstElements{\kappa}} \cap \AGENT^{\player}} } }%
    {{\displaystyle \sum_{ \kappa \in \ORDER, \LastElement{\kappa}=\agent}} 
    \MixedStrategy^{\player} \defset{ \wstrategy\in\WSTRATEGY^{\player} }{ 
    \history\in\HISTORY_{\kappa}^{\ordering} 
    \eqsepv \wstrategy_{\range{\FirstElements{\kappa}} \cap \AGENT^{\player}}\np{\history} 
    = \history_{\range{\FirstElements{\kappa}} \cap \AGENT^{\player}} } }
    \nonumber
\\
  &=
    \frac{ \sum_{ \kappa \in \ORDER, \LastElement{\kappa}=\agent}
    \MixedStrategy^{\player}\ba{\Phi_\kappa\np{\history,\control_{\agent}}} }%
    { \sum_{ \kappa \in \ORDER, \LastElement{\kappa}=\agent} 
    \sum_{  \control_{\agent}\in\CONTROL_{\agent} }
    \MixedStrategy^{\player}\ba{\Phi_\kappa\np{\history,\control_{\agent}}} }
    \eqfinv
    \nonumber
    \intertext{where, for any \( \history\in\HISTORY \),
    \( \control_{\agent}\in\CONTROL_{\agent} \) and \( \kappa \in \ORDER \) such that 
    \( \LastElement{\kappa}=\agent \), we have defined the following subset of strategies}
    \Phi_\kappa\np{\history,\control_{\agent}} 
  &=
    \label{eq:proof_Phi_kappa}
  \\
  &
    \defset{ \wstrategy\in\WSTRATEGY^{\player} }{
    \history\in\HISTORY_{\kappa}^{\ordering} 
    \eqsepv \wstrategy_{\agent}\np{\history}=\control_{\agent}
    \eqsepv \wstrategy_{\range{\FirstElements{\kappa}} \cap \AGENT^{\player}}\np{\history} 
    = \history_{\range{\FirstElements{\kappa}} \cap \AGENT^{\player}} } 
    \subset \WSTRATEGY^{\player} 
    \eqfinp
    \nonumber
\end{align}
%
We will prove, in three steps, that \( \Phi_{\kappa}\np{\history,\control_{\agent}} \) in~\eqref{eq:proof_Phi_kappa} 
takes the same (set) value for any \( \history \in \atom_{\agent} \),
where $\atom_{\agent}$ is an atom of~\( \tribu{\Information}_{\agent} \).
\medskip

\noindent$\bullet$
Let $\atom_{\agent} \subset \HISTORY$ be an atom of~\( \tribu{\Information}_{\agent} \).
We prove that there exists a unique \( \kappa_0 \in \ORDER \) 
such that \( \LastElement{\kappa_0}=\agent \)
and \( \atom_{\agent} \subset \HISTORY_{\kappa_0}^{\ordering} \),
that is, we prove that 
\begin{equation}
\atom_{\agent} \in \crochet{ \tribu{\Information}_{\agent} }
\implies \exists ! \, \kappa_0 \in \ORDER \eqsepv
\LastElement{\kappa_0}=\agent \eqsepv 
\atom_{\agent} \subset \HISTORY_{\kappa_0}^{\ordering}
\eqfinp
  \label{eq:proof_unique_kappa_0}
\end{equation}

First, we show that, for any \( \kappa \in \ORDER \) such that 
\( \LastElement{\kappa}=\agent \), we have 
either \( \atom_{\agent} \subset \HISTORY_{\kappa}^{\ordering} \)
or \( \atom_{\agent} \cap \HISTORY_{\kappa}^{\ordering} = \emptyset \).
Indeed, by~\eqref{eq:PerfectRecall} with 
\( \History =  \HISTORY \in \tribu{\Choice}_{\range{\FirstElements{\kappa}} \cap \AGENT^{\player}}\), 
we obtain that \( \HISTORY_{\kappa}^{\ordering} 
\in \tribu{\Information}_{\agent} \).
Therefore, either \( \atom_{\agent} \cap \HISTORY_{\kappa}^{\ordering} = \emptyset \) 
or \( \atom_{\agent} \subset \HISTORY_{\kappa}^{\ordering} \)
by~\eqref{eq:atom_property} since
$\atom_{\agent}$ is an atom of~\( \tribu{\Information}_{\agent} \).
Second, we have seen that the family \( \sequence{ \HISTORY_{\kappa}^{\ordering} }{
\kappa \in \ORDER, \, \LastElement{\kappa}=\agent } \) is made of
disjoint sets whose union is~\( \HISTORY \).

By combining both results, we conclude that there exists a unique \( \kappa_0 \in \ORDER \) 
such that \( \LastElement{\kappa_0}=\agent \)
and \( \atom_{\agent} \subset \HISTORY_{\kappa_0}^{\ordering} \).
\medskip

\noindent$\bullet$
As a consequence of~\eqref{eq:proof_unique_kappa_0}, we have that 
\( \Phi_\kappa\np{\history,\control_{\agent}} =\emptyset \),
for any \( \history\in \atom_{\agent} \), and 
for any \( \kappa \in \ORDER \) such that \( \kappa \neq \kappa_0 \)
and  \( \LastElement{\kappa}=\agent \).
There only remains to prove that \( \Phi_{\kappa_0}\np{\history,\control_{\agent}} \) 
in~\eqref{eq:proof_Phi_kappa} takes the same (set) value
for any \( \history \in \atom_{\agent} \). For this purpose, 
we consider $\history', \history'' \in \HISTORY $ which belong to 
the atom~$\atom_{\agent}  \in \crochet{\tribu{\Information}_{\agent}}$, that is,
\( \{\history', \history''\} \subset \atom_{\agent} \),
and we establish two preliminary results.

First, we prove that 
\( \history'_{\range{\FirstElements{\kappa_0}} \cap \AGENT^{\player}}=
\history''_{\range{\FirstElements{\kappa_0}} \cap \AGENT^{\player}} \).
For this purpose, we define the subset 
\( \History'=\defset{ \history \in \HISTORY }{
\history_{\range{\FirstElements{\kappa_0}} \cap \AGENT^{\player}}
=\history'_{\range{\FirstElements{\kappa_0}} \cap \AGENT^{\player}} } 
\subset \HISTORY \) and we show in two steps that 
\( \history'' \in \History' \), hence that 
\( \history'_{\range{\FirstElements{\kappa_0}} \cap \AGENT^{\player}}=
\history''_{\range{\FirstElements{\kappa_0}} \cap \AGENT^{\player}} \), that is,
we show that 
\begin{equation}
\{\history', \history''\} \subset \atom_{\agent} 
\implies
 \history'_{\range{\FirstElements{\kappa_0}} \cap \AGENT^{\player}}=
\history''_{\range{\FirstElements{\kappa_0}} \cap \AGENT^{\player}}
\eqfinp
    \label{eq:proof_recall_decisions}
\end{equation}

\begin{description}
\item[-]
We show that 
\( \HISTORY_{\kappa_0}^{\ordering} \cap \History' \in \tribu{\Information}_{\agent} \).
By definition of the field~\( \tribu{\Control}_{\range{\FirstElements{\kappa_0}}
  \cap \AGENT^{\player}} \) in~\eqref{eq:sub_control_field_BGENT} 
with \( \BGENT = \range{\FirstElements{\kappa}} \cap \AGENT^{\player} \),
and because each field~\( \tribu{\Control}_{\bgent} \), for
\( \bgent \in \range{\FirstElements{\kappa_0}} \cap \AGENT^{\player} \), is complete,
hence has the singletons for atoms, 
we have that \( \History' \in \tribu{\Control}_{\range{\FirstElements{\kappa_0}} \cap \AGENT^{\player}} \).
As \( \tribu{\Control}_{\range{\FirstElements{\kappa_0}} \cap \AGENT^{\player}} 
\subset 
\tribu{\Choice}_{\range{\FirstElements{\kappa_0}} \cap \AGENT^{\player}} \)
by~\eqref{eq:PastChoiceField}, we use the perfect recall
assumption~\eqref{eq:PerfectRecall}
with \( \History = \History' \),
and obtain that \( \HISTORY_{\kappa_0}^{\ordering} \cap \History' 
\in \tribu{\Information}_{\LastElement{\kappa_0}}=\tribu{\Information}_{\agent}
\) since \( \LastElement{\kappa_0}=\agent \).
\item[-]
As \( \HISTORY_{\kappa_0}^{\ordering} \cap \History' \in
\tribu{\Information}_{\agent} \)
and $\atom_{\agent}$ is an atom of~\( \tribu{\Information}_{\agent} \),
we have either 
\( \atom_{\agent} \cap \HISTORY_{\kappa_0}^{\ordering} \cap \History' =\emptyset \)
or \( \atom_{\agent} \subset \HISTORY_{\kappa_0}^{\ordering} \cap \History' \)
by~\eqref{eq:atom_property}.
Now, as \( \history' \in \History' \) and 
\( \history' \in \atom_{\agent} \subset \HISTORY_{\kappa_0}^{\ordering} \),
we get that 
\( \atom_{\agent} \cap \HISTORY_{\kappa_0}^{\ordering} \cap \History' \neq \emptyset \),
and therefore \( \atom_{\agent} \subset \HISTORY_{\kappa_0}^{\ordering} \cap \History' \).
Since \( \history'' \in \atom_{\agent} \), we deduce that 
\( \history'' \in \History' \), hence that 
\( \history'_{\range{\FirstElements{\kappa_0}} \cap \AGENT^{\player}}=
\history''_{\range{\FirstElements{\kappa_0}} \cap \AGENT^{\player}} \).
\end{description}

Second, we prove that, for any \( \bgent \in \range{\FirstElements{\kappa_0}} \cap \AGENT^{\player} \)
and any atom $\atom_\bgent$ of~\( \tribu{\Information}_{\bgent} \),
we have either 
\( \{\history', \history''\} \cap \atom_\bgent =\emptyset\)
or \( \{\history', \history''\} \subset \atom_\bgent \), that is, 
we prove that
\begin{equation}
  \begin{split}
\Bp{ \bgent \in \range{\FirstElements{\kappa_0}} \cap \AGENT^{\player} 
\text{ and }
\atom_\bgent \in \crochet{ \tribu{\Information}_{\bgent} } }
\implies \\
\Bp{ \{\history', \history''\} \cap \atom_\bgent =\emptyset
\text{ or }
\{\history', \history''\} \subset \atom_\bgent }
\eqfinp
  \end{split}
    \label{eq:proof_recall_information}
\end{equation}
For this purpose, we consider 
\( \bgent \in \range{\FirstElements{\kappa_0}} \cap \AGENT^{\player} \)
and \( \atom_\bgent \in \crochet{ \tribu{\Information}_{\bgent} } \).

As \( \tribu{\Information}_{\bgent} \subset 
\tribu{\Choice}_{\range{\FirstElements{\kappa_0}} \cap \AGENT^{\player}} \)
by~\eqref{eq:PastChoiceField}, we use the perfect recall
assumption~\eqref{eq:PerfectRecall}
with \( \History = \atom_\bgent \),
and obtain that \( \HISTORY_{\kappa_0}^{\ordering} \cap \atom_\bgent
\in \tribu{\Information}_{\LastElement{\kappa_0}}=\tribu{\Information}_{\agent}
\) since \( \LastElement{\kappa_0}=\agent \) by~\eqref{eq:proof_unique_kappa_0}.
As a consequence, as \( \HISTORY_{\kappa_0}^{\ordering} \cap \atom_\bgent
\in \tribu{\Information}_{\agent} \)
and $\atom_{\agent}$ is an atom of~\( \tribu{\Information}_{\agent} \),
we have either 
\( \atom_{\agent} \cap \HISTORY_{\kappa_0}^{\ordering} \cap \atom_\bgent =\emptyset \)
or \( \atom_{\agent} \subset \HISTORY_{\kappa_0}^{\ordering} \cap \atom_\bgent \)
by~\eqref{eq:atom_property}.
Since \( \{\history', \history''\} \subset \atom_{\agent} \) by assumption,
where \( \atom_{\agent} \subset \HISTORY_{\kappa_0}^{\ordering} \)
by~\eqref{eq:proof_unique_kappa_0}, we conclude
that either \( \{\history', \history''\} \subset \atom_\bgent \)
or \( \{\history', \history''\} \cap \atom_\bgent =\emptyset\).
\medskip

\noindent$\bullet$
We consider an atom~$\atom_{\agent} \in \crochet{\tribu{\Information}_{\agent}}$,
and we finally prove that \( \Phi_{\kappa_0}\np{\history,\control_{\agent}} \) 
in~\eqref{eq:proof_Phi_kappa} takes the same (set) value
for any \( \history \in \atom_{\agent} \).
For this purpose, we consider $\history'$ and $\history''$ which belong to 
the atom~$\atom_{\agent}$ of~\( \tribu{\Information}_{\agent} \), that is,
\( \{\history', \history''\} \subset \atom_{\agent} \),
and we show that 
\( \wstrategy\in\Phi_{\kappa_0}\np{\history',\control_{\agent}} \implies
\wstrategy\in\Phi_{\kappa_0}\np{\history'',\control_{\agent}} \).

Thus, we take \( \wstrategy\in\Phi_{\kappa_0}\np{\history',\control_{\agent}} \)
in~\eqref{eq:proof_Phi_kappa} --- that is, \(\wstrategy\in\WSTRATEGY^{\player} \) is such that 
\( \history'\in\HISTORY_{\kappa_0}^{\ordering} \), 
\( \wstrategy_{\agent}\np{\history'}=\control_{\agent} \) 
and \( \wstrategy_{\range{\FirstElements{\kappa_0}} \cap \AGENT^{\player}}\np{\history'} 
= \history'_{\range{\FirstElements{\kappa_0}} \cap \AGENT^{\player}} \) ---
and we are going to prove in three steps that 
\( \history''\in\HISTORY_{\kappa_0}^{\ordering} \), 
\( \wstrategy_{\agent}\np{\history''}=\control_{\agent} \) 
and \( \wstrategy_{\range{\FirstElements{\kappa_0}} \cap \AGENT^{\player}}\np{\history''} 
= \history''_{\range{\FirstElements{\kappa_0}} \cap \AGENT^{\player}} \).
\begin{description}
\item[-]
First, we have that 
\( \history''\in\HISTORY_{\kappa_0}^{\ordering} \)
since 
\( \{\history', \history''\} \subset \atom_{\agent} \) by assumption,
and \( \atom_{\agent} \subset \HISTORY_{\kappa_0}^{\ordering} \) 
by~\eqref{eq:proof_unique_kappa_0}.
\item[-]
Second, since \( \{\history', \history''\} \subset \atom_{\agent} \) by assumption
and since the strategy \( \wstrategy_{\agent} \) 
is \( \tribu{\Information}_{\agent} \)-measurable, we have that
\( \wstrategy_{\agent}\np{\history'}=\wstrategy_{\agent}\np{\history''} \)
by~\eqref{eq:strategy_atoms_rho}, 
hence \( \wstrategy_{\agent}\np{\history''}=\control_{\agent} \) since
\( \wstrategy_{\agent}\np{\history'}=\control_{\agent} \) by assumption.
\item[-]
Third, we have shown that, for any \( \bgent \in \range{\FirstElements{\kappa_0}} \cap \AGENT^{\player} \)
and any atom $\atom_\bgent$ of~\( \tribu{\Information}_{\bgent} \),
we have either \( \{\history', \history''\} \subset \atom_\bgent \)
or \( \{\history', \history''\} \cap \atom_\bgent =\emptyset\)
by~\eqref{eq:proof_recall_information}.
As a consequence, the pair \( \{\history', \history''\} \)
is included in one of the atoms that make 
the partition~\( \crochet{\tribu{\Information}_{\bgent}} \).
Therefore, we obtain that 
\( \wstrategy_{\bgent}\np{\history'} =
\wstrategy_{\bgent}\np{\history''} \) by~\eqref{eq:strategy_atoms_rho},
hence \( \wstrategy_{\range{\FirstElements{\kappa_0}} \cap \AGENT^{\player}}\np{\history'} =
\wstrategy_{\range{\FirstElements{\kappa_0}} \cap \AGENT^{\player}}\np{\history''}
\), using the notation~\eqref{eq:sub_wstrategy_BGENT} for \( \wstrategy_\BGENT \), 
with \( \BGENT = \range{\FirstElements{\kappa}} \cap \AGENT^{\player} \).
As we had obtained in~\eqref{eq:proof_recall_decisions}
that \( \history'_{\range{\FirstElements{\kappa_0}} \cap \AGENT^{\player}}=
\history''_{\range{\FirstElements{\kappa_0}} \cap \AGENT^{\player}} \), we conclude that 
\( \wstrategy_{\range{\FirstElements{\kappa_0}} \cap \AGENT^{\player}}\np{\history''} =
\wstrategy_{\range{\FirstElements{\kappa_0}} \cap \AGENT^{\player}}\np{\history'} =
\history'_{\range{\FirstElements{\kappa_0}} \cap \AGENT^{\player}}=
\history''_{\range{\FirstElements{\kappa_0}} \cap \AGENT^{\player}} \).
\end{description}
As a consequence, we have just proved that
\( \wstrategy\in\Phi_{\kappa_0}\np{\history'',\control_{\agent}} \),
hence that 
\( \Phi_{\kappa_0}\np{\history',\control_{\agent}} =
\Phi_{\kappa_0}\np{\history'',\control_{\agent}} \) 
whenever 
\( \{\history', \history''\} \subset \atom_{\agent} \).
\medskip

\noindent$\bullet$
Let $\atom_{\agent}$ be an atom of~\( \tribu{\Information}_{\agent} \).
Finally, since \( \Phi_{\kappa}\np{\history,\control_{\agent}} \) 
in~\eqref{eq:proof_Phi_kappa} takes the same (set) value
for any \( \history \in \atom_{\agent} \), 
the expression~\eqref{eq:MtoB_PerfectRecall} takes the same value
for any \( \history \in \atom_{\agent} \), and thus the function
\( \history \in \HISTORY \mapsto 
\MtoB{\MixedStrategy}^{\player}_{\agent}
\conditionaly{\control_{\agent}}{\history} \) is 
\( \tribu{\Information}_{\agent} \)-measurable.
\medskip

This ends the proof.
\end{proof}
\end{subequations}

\subsection{Kuhn's equivalence theorem for causal finite games in intrinsic form}
\label{Main_result}

Finally, we give a statement and a proof of Kuhn's equivalence theorem for games in intrinsic form. 

\begin{theorem}
We consider a causal finite W-game
(see Definition~\ref{de:W-game}),
and a player \( \player \in \PLAYER \) 
supposed to enjoy perfect recall (see Definition~\ref{de:PerfectRecall}).

Then, for any mixed W-strategy \( \MixedStrategy^{\player} \in
\Delta\np{\WSTRATEGY^{\player}} \),
there exists a product-mixed W-strategy
\( \ProductMixedStrategy^{\player}=
\sequence{ \ProductMixedStrategy^{\player}_{\agent}}{\agent \in
  \AGENT^{\player}}  \in \prod_{\agent \in \AGENT^{\player}} \Delta\np{\WSTRATEGY_{\agent}} \), 
as in Definition~\ref{de:product-mixed_W-strategy}, 
such that\footnote{%
See Footnote~\ref{ft:product-mixed_W-strategy}
for the abuse of notation \( \ProductMixedStrategy^{\player}=
\otimes_{\agent\in\AGENT^{\player}}\ProductMixedStrategy^{\player}_{\agent} \).}
\begin{equation}
 \QQ^{\omega}_{\couple{{\MixedStrategy}^{-\player}}{{\MixedStrategy}^{\player}}} 
= \QQ^{\omega}_{\couple{{\MixedStrategy}^{-\player}}{{\ProductMixedStrategy}^{\player}}}
\eqsepv \forall {\MixedStrategy}^{-\player} 
\in \prod_{\player'\neq \player} \Delta\bp{ \WSTRATEGY^{\player'} } 
\eqsepv \forall \omega\in\Omega
\eqfinv
\end{equation}
where the probability distribution \( \QQ^{\omega}_{\MixedStrategy} 
\in \Delta\bp{\prod_{\bgent \in \AGENT} \CONTROL_{\bgent}} \)
has been defined in~\eqref{eq:push_forward_probability}.
%
\label{th:Kuhn}
\end{theorem}


\begin{subequations}
  \begin{proof} 
    The proof is in three steps. 
    \medskip

    \noindent$\bullet$
    First, as all the assumptions of
    Proposition~\ref{pr:MtoB_PerfectRecall} are satisfied,
    there exists a behavioral W-strategy
    \( \MtoB{\MixedStrategy}^{\player} = 
    \sequence{\MtoB{\MixedStrategy}^{\player}_{\agent}}{\agent \in \AGENT^{\player}} \), 
    as in Definition~\ref{de:behavioral_W-strategy},
    which satisfies~\eqref{eq:MtoB_PerfectRecall}.
    By Proposition~\ref{pr:BtoPM}, we define 
    the product-mixed W-strategy
    \( \ProductMixedStrategy^{\player}=
    \BtoPM{\MtoB{\MixedStrategy}}^{\player} \), that is,
    with the property~\eqref{eq:BtoPM} that, for any agent~\( \agent\in\AGENT^{\player} \), 
    \begin{equation}
      \ProductMixedStrategy^{\player}_{\agent}
      \defset{ \wstrategy_{\agent}\in\WSTRATEGY_{\agent}}%
      {\wstrategy_{\agent}\np{\history}=\control_{\agent}}
      = 
      \MtoB{\MixedStrategy}^{\player}_{\agent}
      \conditionaly{\control_{\agent}}{\history}
      \eqsepv \forall \control_{\agent}\in\CONTROL_{\agent} 
      \eqsepv \forall \history \in \HISTORY 
      \eqfinp
      \label{eq:Kuhn_proof_BtoPM}
    \end{equation}
    \medskip

    \noindent$\bullet$
    Second, we prove that\footnote{%
      See Footnote~\ref{ft:product-mixed_W-strategy}
      for the abuse of notation \( \ProductMixedStrategy^{\player}=
      \otimes_{\agent\in\AGENT^{\player}}\ProductMixedStrategy^{\player}_{\agent} \).}
    \begin{align}
      \MixedStrategy^{\player}
      & \defset{ \sequence{\wstrategy_{\agent}}{\agent\in\AGENT^{\player} }
        \in\WSTRATEGY^{\player} }%
        { \lambda_{\agent}\np{\history} =\history_{\agent}
        \eqsepv \forall \agent\in \AGENT^{\player} }
        \nonumber
      \\
      &=
          \ProductBehavioral^{\player}
          \defset{ \sequence{\wstrategy_{\agent}}{\agent\in\AGENT^{\player} }
          \in\WSTRATEGY^{\player} }%
          { \lambda_{\agent}\np{\history} =\history_{\agent}
          \eqsepv \forall \agent\in \AGENT^{\player} }
          \eqsepv \forall \history \in \HISTORY 
          \eqfinp  
        \label{eq:proof_Kuhn_mixedstrategy_ProductBehavioral}
    \end{align}
    In what follows, we consider a configuration~\( \history\in\HISTORY \)
    and the total ordering~\( \totalordering=\ordering\np{\history} \in
    \ORDER^{\numAGENT} \).
    We label the set~$\AGENT^{\player}$ of agents of the player~$\player$ 
    by the stage at which each of them plays as follows:
    \begin{equation}
      \AGENT^{\player}=\{ \totalordering(j_1) ,\ldots, \totalordering(j_{N}) \} 
      \text{  with } j_1 < \cdots < j_{N}
      \eqfinp  
      \label{eq:labelling}
    \end{equation}
    With this, we have\footnote{Using the notation $\ic{n}=\na{1,\ldots,n}$ to shorten some expressions.}
    \begin{align*}
      \MixedStrategy^{\player}
      &
        \defset{ \sequence{\wstrategy_{\agent}}{\agent\in\AGENT^{\player} }
        \in\WSTRATEGY^{\player} }%
        { \lambda_{\agent}\np{\history} =\history_{\agent}
        \eqsepv \forall \agent\in \AGENT^{\player} }
      \\
      &=
        \MixedStrategy^{\player}
        \defset{ \sequence{\wstrategy_{\totalordering(j_k)}}{k\in\ic{N}}
        \in \prod_{k=1}^{N} \WSTRATEGY_{\totalordering(j_k)} }%
        { \lambda_{\totalordering(j_k)}\np{\history} =\history_{\totalordering(j_k)}
        \eqsepv \forall k\in \ic{N}}
        \tag{as \( \AGENT^{\player}=\{ \totalordering(j_1) ,\ldots,
        \totalordering(j_{N}) \} \)
        by~\eqref{eq:labelling}}
      \\
      &=
        \prod_{n=1}^{N}
        \frac{
        \MixedStrategy^{\player}
        \defset{ \sequence{\wstrategy_{\totalordering(j_k)}}{k\in\ic{n}}
        \in {\displaystyle \mathop{\prod}_{k=1}^{n}} \WSTRATEGY_{\totalordering(j_k)} }%
        { \lambda_{\totalordering(j_k)}\np{\history} =\history_{\totalordering(j_k)}
        \eqsepv \forall k \in \ic{n}}
        }{%
        \MixedStrategy^{\player}
        \defset{ \sequence{\wstrategy_{\totalordering(j_k)}}{k\in\ic{n{-}1}}
        \in {\displaystyle \prod_{k=1}^{n-1}}\WSTRATEGY_{\totalordering(j_k)} }%
        { \lambda_{\totalordering(j_k)}\np{\history} =\history_{\totalordering(j_k)}
        \eqsepv \forall k\in\ic{n-1}}
        }
        \intertext{where, if the smaller term (the one to be found two equality lines above) is zero, every fraction
        is supposed to take the value zero, and, if the smaller term is positive, so
        are all the terms and no denominator is zero} 
      &=
        \prod_{n=1}^{N} 
        \MtoB{\MixedStrategy}^{\player}_{\totalordering(j_n)}%
        \conditionaly{\history_{\totalordering(j_n)}}{\history}
        \intertext{by~\eqref{eq:MtoB_PerfectRecall}, because 
        \( \range{\cut_{j_n}\np{\totalordering}} \cap \AGENT^{\player}
        =
        \{ \totalordering(j_n) \} \cup \bp{
        \range{\cut_{j_{n-1}}\np{\totalordering}} \cap \AGENT^{\player} } \) 
        by definition~\eqref{eq:cut} of the restriction mapping~$\cut$,
        and by definition of the sequence \( j_1 < \cdots < j_{N} \) 
        in~\eqref{eq:labelling}, which is such that 
        \( \AGENT^{\player}=\{ \totalordering(j_1) ,\ldots, \totalordering(j_{N}) \} \)} 
      &=
        \prod_{n=1}^{N} 
        \ProductBehavioral^{\player}_{\totalordering(j_n)}
        \defset{ \wstrategy_{\totalordering(j_n)} \in \WSTRATEGY_{\totalordering(j_n)} }%
        {\wstrategy_{\totalordering(j_n)} \np{\history}=\history_{\totalordering(j_n)} }
        \tag{by~\eqref{eq:Kuhn_proof_BtoPM}}
      \\
      &=
        \Bp{ \bigotimes_{n=1}^{N} 
        \ProductBehavioral^{\player}_{\totalordering(j_n)} }
        \defset{ \sequence{\wstrategy_{\totalordering(j_k)}}{k\in\ic{N}} 
        \in \prod_{k=1}^{N} \WSTRATEGY_{\totalordering(j_k)} }%
        { \lambda_{\totalordering(j_k)}\np{\history} =\history_{\totalordering(j_k)}
        \eqsepv \forall k\in\ic{N}}
        \tag{by definition of the product probability}
      \\
      &=
        \Bp{ \bigotimes_{n=1}^{N} 
        \ProductBehavioral^{\player}_{\totalordering(j_n)} }
        \defset{ \sequence{\wstrategy_{\agent}}{\agent\in\AGENT^{\player} }
        \in\WSTRATEGY^{\player} }%
        { \lambda_{\agent}\np{\history} =\history_{\agent} }
        \tag{as \( \AGENT^{\player}=\{ \totalordering(j_1) ,\ldots,
        \totalordering(j_{N}) \} \) in~\eqref{eq:labelling}}
    \end{align*}
    Thus, we have proved~\eqref{eq:proof_Kuhn_mixedstrategy_ProductBehavioral}.
    \medskip

    \noindent$\bullet$
    Third, for any configuration~\( \history=\bp{\omega, 
      \sequence{\control_\bgent}{\bgent\in\AGENT} } \in\HISTORY \), we have 

    \begin{align*}
      \QQ^{\omega}_{\MixedStrategy}
      \bp{\sequence{\control_\bgent}{\bgent\in\AGENT} }
      & =
        \bp{\bigotimes_{\player \in \PLAYER}\MixedStrategy^{\player}}
        \Bp{ \ReducedSolutionMap\np{\omega,\cdot}^{-1}
        \bp{\sequence{\control_\bgent}{\bgent\in\AGENT} } } 
        \tag{by definition~\eqref{eq:push_forward_probability_a} 
        of~\( \QQ^{\omega}_{\MixedStrategy} \) }
      \\
      &=
        \bp{ \bigotimes_{\player \in \PLAYER}\MixedStrategy^{\player} }
        \defset{\wstrategy\in\WSTRATEGY}{ \ReducedSolutionMap\np{\omega,\wstrategy} 
        = \sequence{\control_\bgent}{\bgent\in\AGENT} }
        \tag{by definition of a pushforward probability} 
      \\
      &=
        \bp{ \bigotimes_{\player \in \PLAYER}\MixedStrategy^{\player} }
        \defset{\wstrategy\in\WSTRATEGY}{ \lambda_{\agent}\bp{\omega, 
        \sequence{\control_\bgent}{\bgent\in\AGENT} } =\control_{\agent}
        \eqsepv \forall \agent\in \AGENT }
        \tag{by~\eqref{eq:push_forward_probability_b} 
        and~\eqref{eq:solution_map_IFF} which define
        the mapping~\( \ReducedSolutionMap\np{\omega,\cdot}\)} 
      \\
      &=
        \prod_{\player \in \PLAYER} \MixedStrategy^{\player}
\defset{ \sequence{\wstrategy_{\agent}}{\agent\in\AGENT^{\player} }
        \in\WSTRATEGY^{\player} }%
        { \lambda_{\agent}\bp{\omega, 
        \sequence{\control_\bgent}{\bgent\in\AGENT} } =\control_{\agent}
        \eqsepv \forall \agent\in \AGENT^{\player} } 
        \intertext{by definition of the product probability 
        \( \bigotimes_{\player \in \PLAYER}\MixedStrategy^{\player} \)
        on the product space
        \( \prod_{\player \in \PLAYER} \WSTRATEGY^{\player} \) }
      &=
        \prod_{\player \in \PLAYER} \MixedStrategy^{\player}
\defset{ \sequence{\wstrategy_{\agent}}{\agent\in\AGENT^{\player} }
        \in\WSTRATEGY^{\player} }%
        { \lambda_{\agent}\np{\history}=\history_{\agent}
        \eqsepv \forall \agent\in \AGENT^{\player} } 
        \intertext{(because \( \history=\bp{\omega, 
        \sequence{\control_\bgent}{\bgent\in\AGENT} } \) 
        and where we have used notation~\eqref{eq:history})}
      &=
        \prod_{\player' \neq \player} \MixedStrategy^{\player'}
\defset{ \sequence{\wstrategy_{\agent}}{\agent\in\AGENT^{\player'} }
        \in\WSTRATEGY^{\player'} }%
        { \lambda_{\agent}\np{\history} =\history_{\agent}
        \eqsepv  \forall \agent\in \AGENT^{\player'} } 
      \\
      &\phantom{==} \times
        \MixedStrategy^{\player}
\defset{ \sequence{\wstrategy_{\agent}}{\agent\in\AGENT^{\player} }
        \in\WSTRATEGY^{\player} }%
        { \lambda_{\agent}\np{\history} =\history_{\agent}
        \eqsepv  \forall \agent\in \AGENT^{\player} } 
        \tag{where we have singled out the player~$\player$}
      \\
      &=
        \prod_{\player' \neq \player} \MixedStrategy^{\player'}
\defset{ \sequence{\wstrategy_{\agent}}{\agent\in\AGENT^{\player'} }
        \in\WSTRATEGY^{\player'} }%
        { \lambda_{\agent}\np{\history} =\history_{\agent}
        \eqsepv  \forall \agent\in \AGENT^{\player'} } 
      \\
      &\phantom{==} \times
        \ProductBehavioral^{\player}
\defset{ \sequence{\wstrategy_{\agent}}{\agent\in\AGENT^{\player} }
        \in\WSTRATEGY^{\player} }%
        { \lambda_{\agent}\np{\history} =\history_{\agent}
        \eqsepv  \forall \agent\in \AGENT^{\player} } 
        \tag{by~\eqref{eq:proof_Kuhn_mixedstrategy_ProductBehavioral}}
      \\
      &=
        \QQ^{\omega}_{\couple{{\MixedStrategy}^{-\player}}{{\ProductMixedStrategy}^{\player}}}
        \bp{\sequence{\control_\bgent}{\bgent\in\AGENT} }
        \eqfinv
    \end{align*}
    by reverting to the top equality with 
    \( \MixedStrategy^{\player} \) replaced by 
    \( \ProductBehavioral^{\player} \). 
    \medskip

    This ends the proof.
  \end{proof}
\end{subequations}

\section{Discussion}

Most games in extensive form are formulated on a tree.
However, whereas trees are perfect to follow step by step how a game is played,
they can be delicate to manipulate when information sets are added
and must satisfy restrictive axioms to comply with the tree structure
\cite{Alos-Ferrer-Ritzberger:2016,Bonanno:2004,Brandenburger:2007}.
In this paper, we have introduced the notion of games in intrinsic form,
where the tree structure is replaced with a product structure,
more amenable to mathematical analysis. 
For this, we have adapted Witsenhausen's intrinsic model --- 
a model with Nature, agents and their decision sets, and
where information is represented by $\sigma$-fields --- to games.
In contrast to games in extensive form formulated on a tree,
Witsenhausen's intrinsic games (W-games) do not require
an explicit description of the play temporality. 
Not having a hardcoded temporal ordering makes mathematical representations more intrinsic.

As part of a larger research program, we have focused here on Kuhn's equivalence
theorem.
For this purpose, we have defined the property of perfect recall for a player 
of a causal W-game (that is, without referring to a tree structure),
and we have introduced three different definitions of
``randomized'' strategies in the W-games setting
--- mixed, product-mixed and behavioral. 
Then, we have shown that, under perfect recall for a player,
any of her possible mixed strategies can be replaced by a behavioral strategy,
which is the statement of Kuhn's equivalence theorem.
Moreover, we have shown that any of her possible mixed strategies can also be replaced by a
product-mixed strategy, that is, a mixed strategy under which her executive
agents are probabilistically independent.

We add to the existing literature on extensive games representation 
by proposing a representation that is more general than the tree-based ones as,
for instance, it allows to describe noncausal situations.
Indeed, Witsenhausen showed that there are noncausal W-models that yet are
solvable. 

Furthermore, our paper illustrates that the intrinsic form is well equipped to
handle proofs with mathematical formulas, without resorting to tree-based
  arguments that can be cumbersome when handling information. 
We hence believe that the intrinsic form constitutes a new valuable tool for the analysis of games 
with information. 

The current work is the first output of a larger research program 
that addresses games in intrinsic form.
We are currently working on the embedding of tree-based games in extensive
form into W-games (by a mapping that associates each information set with an
agent), and on the restricted class of W-games that can be 
embedded in tree-based games.
Futher research includes extensions to measurable decision sets, 
and to infinite number of agents or players.
We will also investigate what can be said about subgame perfect equilibria and
backward induction, as well as Bayesian games.
\medskip

\textbf{Acknowledgements}.

We thank Dietmar Berwanger and Tristan Tomala for
fruitful discussions, and for their valuable comments on a first version of this
paper. This research benefited from the support of the FMJH Program PGMO and
from the support to this program from EDF.

\appendix

\section{Background on fields, atoms and partitions}
\label{Background_on_fields_atoms_and_partitions}

\renewcommand{\atom}{\Set}

In this paper, we present the intrinsic model of Witsenhausen
\cite{Witsenhausen:1975,Carpentier-Chancelier-Cohen-DeLara:2015}
but with finite sets rather than with infinite ones 
as in the original exposition.
Witsenhausen used the language of $\sigma$-fields\footnote{%
Recall that a $\sigma$-field over the set~$\SET$ is a subset $\tribu{\Set} \subset 2^\SET$,
containing~$\SET$, and which is stable under complementation 
and countable union.}
to handle the concept of information in control theory.
Because sets are finite, 
we consider a restricted subclass of $\sigma$-fields.
In what follows, $\SET$ is a finite set.

An \emph{algebra}, or a \emph{field}, over the finite set~$\SET$ is
a subset $\tribu{\Set} \subset 2^\SET$,
containing~$\SET$, and which is stable under complementation and union.
The trivial field over the finite set~$\SET$ is the field 
\( \{ \emptyset, \SET \} \).
The \emph{complete field} over the finite set~$\SET$ is the field~$2^\SET$.

An \emph{atom} of the field~$\tribu{\Set}$ (over the finite set~$\SET$)
is a minimal element for the inclusion~$\subset$, that is, an atom 
is a nonempty subset $\atom \in \tribu{\Set}$ 
such that $K \in \tribu{\Set}$ and $K \subset \atom$ imply that $K = \emptyset$ or
$K=\atom$.
We denote by \( \crochet{\tribu{\Set}} \) the 
set of atoms of the field~$\tribu{\Set}$:
\begin{equation}
  \crochet{\tribu{\Set}} = \defset{ \atom \in \tribu{\Set}\backslash\{\emptyset\} }{%
    \bp{ K \in \tribu{\Set} \text{ and } K \subset \atom }
    \Rightarrow
    \bp{ K = \emptyset \text{ or } K=\atom } }
  \eqfinp
\label{eq:atom_set}
\end{equation}
For instance, a complete field has the singletons for atoms.
It can be shown that the atoms of~$\tribu{\Set}$ form a partition of $\SET$, that is,
they consist of mutually disjoint nonempty subsets whose union is $\SET$
\cite[Proposition~3.18]{Carpentier-Chancelier-Cohen-DeLara:2015}.
As a consequence, any element of the field~$\tribu{\Set}$ is necessarily written as 
the union of atoms that it contains, and we have the useful property
\begin{equation}
  \bp{ K \in \tribu{\Set} \text{ and } \atom \in \crochet{\tribu{\Set}} }
  \;\Rightarrow\;
  \bp{ K \cap \atom = \emptyset \text{ or } K \subset \atom }
  \eqfinp
\label{eq:atom_property}
\end{equation}

Consider two fields~$\tribu{\Set}$ and $\tribu{\Set}'$ over the finite set~$\SET$.
We say that the field~$\tribu{\Set}$ is \emph{finer} than the
field~$\tribu{\Set}'$
if \( \tribu{\Set} \supset \tribu{\Set}' \) (notice the reverse inclusion).
We also say that  $\tribu{\Set}'$ is a \emph{subfield} of~$\tribu{\Set}$.
As an illustration, the complete field is finer than any field or,
equivalently, any field is a subfield of the complete field.

The \emph{least upper bound} of two fields~$\tribu{\Set}$ and $\tribu{\Set}'$,
denoted by \( \tribu{\Set} \vee \tribu{\Set}' \),
is the smallest field that contains $\tribu{\Set}$ and $\tribu{\Set}'$.
The atoms of \( \tribu{\Set} \vee \tribu{\Set}' \) are all the nonempty intersections
between an atom of~$\tribu{\Set}$ and an atom of~$\tribu{\Set}'$.
The least upper bound of two fields is finer than any of the two.

Consider a field~$\tribu{\Set}$ over the finite set~$\SET$,
and a field~$\tribu{\Set}'$ over the finite set~$\SET'$.
The \emph{product field} \( \tribu{\Set} \otimes \tribu{\Set}' \)
is the smallest field, over the finite product set~$\SET\times\SET'$,
that contains all the rectangles, that is,
that contains all the products of an element of~$\tribu{\Set}$ with
an element of~$\tribu{\Set}'$. 

\newcommand{\noopsort}[1]{} \ifx\undefined\allcaps\def\allcaps#1{#1}\fi

\end{document}